\documentclass[a4paper,10pt]{article}
\usepackage{amssymb,amsmath,amsthm}
\usepackage{fullpage}
\usepackage{hyperref}
\usepackage{eucal}
\usepackage{color}
\usepackage{stackrel}
\usepackage{graphicx}
\newcommand{\ie}{\emph{i.e.}}
\newcommand{\eg}{\emph{e.g.}}

\newcommand{\cD}{\mathcal{D}}

\newcommand{\Real}{\mathbb{R}}
\newcommand{\C}{\mathbb{C}}

\newcommand{\Dom}{\mathsf{D}}

\newcommand{\sii}{L^2}

\newcommand{\der}{\mathrm{d}}

\DeclareMathOperator{\Sp}{Sp}

\newenvironment{psmallmatrix}
  {\left(\begin{smallmatrix}}
  {\end{smallmatrix}\right)}
\begingroup
    \makeatletter
    \@for\theoremstyle:=definition,remark,plain\do{%
        \expandafter\g@addto@macro\csname th@\theoremstyle\endcsname{%
            \addtolength\thm@preskip\parskip
            }%
        }
\endgroup
\newtheorem{Theorem}{Theorem}
\newtheorem{Lemma}{Lemma}
\newtheorem{Proposition}{Proposition}
\newtheorem{Corollary}{Corollary}
\newtheorem{Conjecture}{Conjecture}
\theoremstyle{definition}
\newtheorem{Remark}{Remark}

\def\OMIT#1{}
\usepackage[normalem]{ulem}
\definecolor{DarkGreen}{rgb}{0,0.5,0.1} 
\newcommand{\txtD}{\textcolor{DarkGreen}}

\newcommand\soutD{\bgroup\markoverwith
{\textcolor{DarkGreen}{\rule[.5ex]{2pt}{1pt}}}\ULon}
\newcommand\soutP{\bgroup\markoverwith
{\textcolor{blue}{\rule[.5ex]{2pt}{1pt}}}\ULon}
\newcommand{\Hm}[1]{\leavevmode{\marginpar{\tiny%
$\hbox to 0mm{\hspace*{-0.5mm}$\leftarrow$\hss}%
\vcenter{\vrule depth 0.1mm height 0.1mm width \the\marginparwidth}%
\hbox to
0mm{\hss$\rightarrow$\hspace*{-0.5mm}}$\\\relax\raggedright #1}}}

\begin{document}
%
\title{\textbf{\LARGE
Spectral inequality for Dirac right triangles 
}}
\author{Tuyen Vu}
\date{\small 
\begin{quote}
\begin{center}
Department of Mathematics, Faculty of Nuclear Sciences and 
Physical Engineering, \\ Czech Technical University in Prague, 
Trojanova 13, 12000 Prague 2, Czechia. \\
E-mail: thibichtuyen.vu@fjfi.cvut.cz.
\end{center}
\end{quote}
25 February 2023
}
\maketitle
\vspace{-5ex} 
\begin{abstract}
\noindent
We consider the Dirac operator on right triangles,
subject to infinite-mass boundary conditions.
We conjecture that the lowest positive eigenvalue 
is minimised by the isosceles right triangle 
both under the area or perimeter constraints.
We prove this conjecture under extra geometric hypotheses 
relying on a recent approach of Ph.~Briet and D.~Krej{\v{c}}i{\v{r}}{\'i}k
for Dirac rectangles~\cite{2DK}.
\end{abstract}
%


\section{Introduction}
%
One of the most interesting topics in spectral geometry 
is the determination of optimal shapes for eigenvalues of differential operators,
subject to various boundary conditions and geometric constraints.
Probably the most classical and well known situation is that 
of the Laplace operator, subject to Dirichlet boundary conditions:
\begin{equation}\label{Dirichlet} 
\left\{
\begin{aligned}
  -\Delta \psi &= \Lambda \psi
  && \mbox{in} && \Omega 
  \,,
  \\
  \psi &= 0 
  && \mbox{on} && \partial\Omega 
  \,,
\end{aligned}
\right.
\end{equation}
where $\Omega$ is an open set of finite measure.
The celebrated Faber--Krahn inequality states that
the lowest eigenvalue $\Lambda_1 = \Lambda_1(\Omega)$
is minimised by the ball, among all sets of given volume. 
By the classical isoperimetric inequality, 
it follows that the ball is the minimiser under the perimeter constraint too.
The optimality of the ball extends to repulsive Robin boundary conditions,
but it is generally false for attractive Robin boundary conditions~\cite{FK7,AFK}.
The ball is generally not optimal for higher eigenvalues either.
Mathematically, the optimality of the ball is closely related
to the availability of symmetrisation techniques.
We refer to the monographs \cite{Henrot,Henrot2} for a recent survey
of this fascinating spectral-optimisation subject.

By a symmetrisation argument, it is also true that 
the Dirichlet eigenvalue $\Lambda_1$ is minimised 
by the equilateral triangle (respectively, square),
among all triangles (respectively, quadrilaterals)
of a given area or perimeter. 
The analogous problem remains open for general polygons,
see \cite[Sec.~3.3.3]{Henrot} and \cite{Bogosel-Bucur,Indrei}
for a survey and the most recent progresses, respectively.
In general, it also remains open for Robin boundary conditions,
even in the case of triangles~\cite{KLV}.
On the other hand, rectangles (or, more generally, rectangular boxes), 
the very special situation of quadrilaterals,
can be settled by means of the availability of explicit solutions
due to the separation of variables~\cite{Laugesen_2019}.

The classical physical interpretation of~$\Lambda_1$ 
in two dimensions is the square 
of the fundamental frequency of a vibrating membrane with fixed edges.
Alternatively, 
$\Lambda_1$ is the ground-state energy of a non-relativistic
quantum particle constrained to a semiconductor nanostructure
of shape~$\Omega$ by hard-wall boundaries.  
In this paper, we are interested in analogues 
of the aforementioned spectral-optimisation
problems in the relativistic setting.

The relativistic analogue of~\eqref{Dirichlet} 
(relevant for graphene materials, for instance)
is the spectral problem for the Dirac operator,
subject to infinite-mass (also called MIT) boundary conditions
\cite{Benguria-Fournais-Stockmeyer-Bosch_2017b,
Arrizibalaga-LeTreust-Raymond_2017,
LeTreust-Ourmieres-Bonafos_2018, 
Arrizibalaga-LeTreust-Raymond_2018,
Barbaroux-Cornean-LeTreust-Stockmeyer_2019,
Arrizibalaga-LeTreust-Mas-Raymond_2019}.
More specifically, given an open Lipschitz set~$\Omega$ in~$\mathbb{R}^2$,
the relativistic quantum Hamiltonian acts as 
\begin{equation}\label{operator1}
  T := 
  \begin{pmatrix}
    m & -i (\partial_1-i\partial_2) \\
    -i(\partial_1+i\partial_2) & -m
  \end{pmatrix}
  \qquad \mbox{in} \qquad
  \sii(\Omega;\C^2)
  \,,
\end{equation}
while the boundary conditions are encoded in the operator domain 
 \begin{equation}\label{operator2}
  \Dom(T) :=  
  \left\{
  \psi = 
  \begin{psmallmatrix}
  \psi_1 \\ \psi_2 
  \end{psmallmatrix}
  \in W^{1,2}(\Omega;\C^2) : \ \psi_2 = i (n_1 + i n_2) \psi_1
  \mbox{ on } \partial\Omega
  \right\}
  .
\end{equation}
Here the notations $m$ and $n= 
  \begin{psmallmatrix}
  n_1 \\ n_2 
  \end{psmallmatrix}
  :\partial\Omega\to\Real^2$ 
  stand for the non-negative mass of the relativistic \mbox{(quasi-)}particle 
  and the outward unit normal of the set~$\Omega$, respectively. 
The operator~$T$ is self-adjoint, at least if
the boundary~$\partial\Omega$ is $C^2$-smooth    
\cite{Benguria-Fournais-Stockmeyer-Bosch_2017b} 
or if~$\Omega$ is a polygon 
\cite{LeTreust-Ourmieres-Bonafos_2018}
(for a general Lipschitz set, the self-adjointness can be achieved
in a $W^{1/2,2}$ setting \cite{Behrndt}).
As usual in relativistic quantum mechanics, 
the spectrum of~$T$ is not bounded from below.   
However, it is purely discrete if~$\Omega$ is bounded   
and the eigenvalues are symmetrically distributed on the real axis.
Consequently, 
the lowest positive eigenvalue $\lambda_1 = \lambda_1(\Omega)$
of~$T$ can be characterised variationally:
\begin{equation}\label{eq2}  
  \lambda_1(\Omega)^2 
  = \inf_{\stackrel[\psi \not= 0]{}{\psi \in \Dom(T)}} 
  \frac{\|T \psi\|^2}{\|\psi\|^2}  
  \,.
\end{equation}
It is important to stress that, 
because of the exotic boundary conditions,
spinorial structure of the Hilbert space 
and lack of positivity-preserving property,
no symmetrisation techniques are available at this moment.

In analogy with the Faber--Krahn inequality,
the following conjecture is natural to expect
in the relativistic setting.
\begin{Conjecture}\label{Conj}
Given any $m \geq 0$ 
and open Lipschitz set $\Omega \subset \Real^2$,
\begin{center}
$ \lambda_1(\Omega) \geq \lambda_1(\Omega^*)$
 \end{center}
where $\Omega^*$ is the disk of the same area or perimeter as~$\Omega$.
\end{Conjecture}

For massless particles (\ie~$m=0$), the fixed-area part of the conjecture
was explicitly stated 
in~\cite{Antunes-Benguria-Lotoreichik-Ourmires-Bonafos_2021}.
The present general statement can be found in \cite{2DK}. 
The proof of the conjecture was classified as a challenging open problem
in spectral geometry during an AIM workshop in San Jose (USA) 
in 2019 \cite{AIM-2019}.
Unfortunately, despite some partial attempts
\cite{Benguria-Fournais-Stockmeyer-Bosch_2017,
Lotoreichik-Ourmieres_2019,
Antunes-Benguria-Lotoreichik-Ourmires-Bonafos_2021},
including a numerical support, 
the problem remains open. 
 
Because of the complexity of the problem in the general setting,
the authors of~\cite{2DK} considered a rectangular version 
of the conjecture.
More specifically, it is conjectured in~\cite{2DK} 
that~$\lambda_1$ is minimised by the square among 
all rectangles of a fixed area or perimeter.
Surprisingly, even this simplified setting is not resolved
and the authors of~\cite{2DK} managed to prove the conjecture 
under some additional hypotheses only
(roughly, for heavy masses or eccentric rectangles). 
The problem is that the infinite-mass boundary conditions
do not allow for a separation of variables.

In this paper, we continue the study by asking whether 
the isosceles right triangle is the optimal geometry
among all right triangles,
again both under the area or perimeter constraints.
More specifically, let $\Omega_{a,b}$ be the right triangle in~$\Real^2$
defined by the three vertexes 
$O:=(0,0)$, $A:=(a,0)$ and $B:=(0,b)$,
where $a,b$ are any positive numbers,
see Figure~\ref{Fig}. 
Note that the area and perimeter are given by
$\frac{1}{2}ab$ and $a+b + \sqrt{a^2+ b^2}$, respectively.
Define $\lambda_1(a,b) := \lambda(\Omega_{a,b})$.

\begin{Conjecture}\label{Conj.main}
Given any $m \geq 0$,
\begin{enumerate}
\item[\emph{(i)}]
$\lambda_1(a,b) \geq \lambda_1(k,k)$ 
with any $a,b,k >0$ such that $ab= k^2$,
\item[\emph{(ii)}]
$\lambda_1(a,b) \geq \lambda_1(k,k)$
with any $a \in (0,(2+\sqrt{2})k)$ and $b,k >0$ such that
$a+b + \sqrt{a^2 +b^2}= (2+\sqrt{2})k$.
\end{enumerate}
\end{Conjecture}

We have not managed to prove the conjecture in its full generality. 
Following~\cite{2DK}, to get partial results,
we first establish universal lower and upper bounds to $\lambda_1(a,b)$.  
\begin{Theorem}\label{Thm.bounds}
For every $m \geq 0$, one has
\begin{equation*} 
   \frac{\arctan^2(\frac{a}{b} +\sqrt{1+\frac{a^2}{b^2}})}{b^2} + \frac{\arctan^2(\frac{b}{a} +\sqrt{1+\frac{b^2}{a^2}})}{a^2}
  \ \leq \ \lambda_1(a,b)^2 - m^2 \ \leq \
 \frac{5\pi^2}{2} \left(\frac{1}{a^2} + \frac{1}{b^2}\right)
  \,.
\end{equation*}
\end{Theorem}

Note that the upper bound becomes sharp 
in the limit $m \to \infty$ for $a=b$.
Indeed, it is known that $\lambda_1(\Omega)^2 - m^2$ 
converges to the Dirichlet eigenvalue $\Lambda_1(\Omega)$
as $m \to \infty$ 
(see, \eg, \cite{Arrizibalaga-LeTreust-Raymond_2017})
and $\Lambda_1(a,a) = 5\pi^2/a^2$.
In contrast to the one-dimensional spectrum of the operator in \cite{2DK}, that in this paper is not symmetric.
As in~\cite{2DK}, the upper bound is obtained by using a suitable
trial function in~\eqref{eq2}. The lower bound employs 
a Poincar\'e-type inequality for a one-dimensional Dirac problem   
on an interval. 
The latter yields an $m$-dependent (implicit) lower bound,
while the lower bound of Theorem~\ref{Thm.bounds}
is due to an (explicit) uniform estimate
of the closest-to-zero eigenvalue of the one-dimensional problem.
  
As a consequence, we get the following sufficient conditions
which guarantee the validity of Conjecture~\ref{Conj.main}.
\begin{Corollary}\label{Corol}
 Let $k$ be defined as in Conjecture~\ref{Conj.main}
 and $ m\geq 0$. \\
Conjecture~\ref{Conj.main}.(i) holds 
under the following extra hypotheses:
\begin{center}
$
  a \geq 9 \, k $ \quad or \quad $ a \leq \frac{k}{9}
$, \\
\end{center}
%
Conjecture~\ref{Conj.main}.(ii) holds 
under the following extra hypotheses:
\begin{center}
$
\displaystyle
  a \geq  3.5 \, k $ \quad or \quad $ a \leq \frac{k}{9}
$.\\
\end{center}
\end{Corollary}

In other words, Conjecture~\ref{Conj.main} holds true
for sufficiently eccentric right triangles.
  
The paper is organised as follows.  
In Section~\ref{Sec.no}, we derive a formula for the expectation value of the square of the Dirac operator in triangles. 
This formula serves as the foundation for the proof of Theorem~\ref{Thm.bounds}.   
The one-dimensional Poincar\'e-type inequality 
is established in Section~\ref{Sec.well}.
The main results are proved in Section~\ref{Sec.proof}.
The extension of the formula for the expectation value 
of the square of the Dirac operator to planar polygons
can be found in Appendix~\ref{A}.

\section{The square of the Dirac operator in polygons}\label{Sec.no}
%
Recall that our right triangle~$\Omega_{a,b}$ 
is special planar polygon determined
by the three vertices $O:=(0,0)$, $A:=(a,0)$ and $B:=(0,b)$.

\begin{figure}[h]
  \begin{center}
  \includegraphics[scale=3]{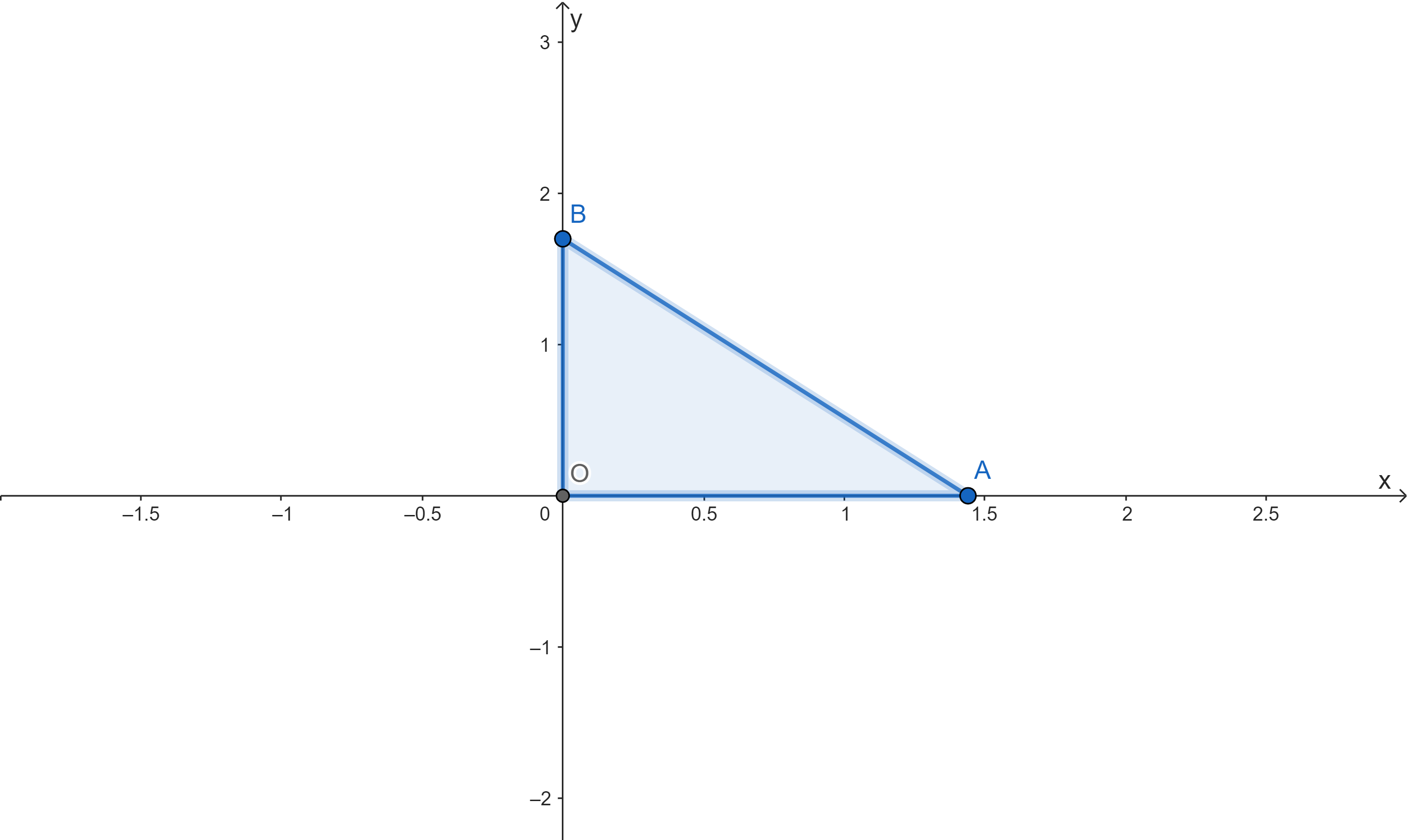}
  \caption{The right triangle $OAB$.
  \\
  }
  \label{Fig}
  \end{center}
\end{figure}

Let~$T_{a,b}$ denote the operator~\eqref{operator1}--\eqref{operator2}
in the case of the triangle~$\Omega_{a,b}$. 
The operator is self-adjoint and has a compact resolvent.
The eigenvalue problem $T_{a,b} u=\lambda u$ 
is equivalent to the system
\begin{equation}\label{system} 
\left\{
\begin{aligned}
  -i(\partial_1-i\partial_2) u_2 &= (\lambda-m) u_1
  && \mbox{in} \quad \Omega_{a,b} \,,
  \\
  -i(\partial_1+i\partial_2) u_1 &= (\lambda+m) u_2
  && \mbox{in} \quad \Omega_{a,b} \,,
  \\
  u_2 &= u_1 
  && \mbox{on} \quad OA \,,
  \\
  u_2 &= -iu_1 
  && \mbox{on} \quad OB \,,
 \\
  u_2 &= \left( \frac{-a}{\sqrt{a^2+b^2}}+ i \, \frac{b}{\sqrt{a^2+b^2}}\right) u_1 
  && \mbox{on} \quad AB \,.
\end{aligned}  
\right.
\end{equation}
The spectrum of~$T_{a,b}$ is symmetric with respect to zero.
Indeed,  
$
  u
  = \begin{psmallmatrix}
  u_1 \\ u_2 
  \end{psmallmatrix}
$ 
is an eigenfunction of~$T_{a,b}$
corresponding to an eigenvalue~$\lambda$ if, and only if,
$
 \begin{psmallmatrix}
  \bar{u}_2 \\ \bar{u}_1 
  \end{psmallmatrix}
$  
is an eigenfunction of~$T_{a,b}$ corresponding to an eigenvalue~$-\lambda$
(charge conjugation symmetry).
It will become evident in a moment
that any solution of~\eqref{system}
necessarily satisfies $|\lambda| \geq m$. Our objective is to study the smallest positive solution, $\lambda_1(a,b)$, 
of~\eqref{system}.

Since the analogous case of rectangles 
cannot be solved by separation of variables~\cite{2DK}, 
there seem to be no hope to get explicit solutions of~\eqref{system}.  
As an alternative approach, we focus on the variational 
characterisation~\eqref{eq2}.
To this aim, we need a more suitable formula for the square norm 
$\|T_{a,b}u\|^2 = (u,T_{a,b}^2 u)$.
If~$\Omega$ were a smooth bounded domain, 
we would have (see, \eg, \cite{Arrizibalaga-LeTreust-Raymond_2017})
\begin{equation}\label{curvature}
\|Tu\|^2 = \|\nabla u\|^2 + m^2 \|u\|^2
+ m \, \|\gamma u\|^2
- \frac{1}{2} \int_{\partial\Omega} \kappa \, |\gamma u|^2 ds 
\end{equation}
for every $u \in \Dom(T)$,
where~$\kappa$ is the signed curvature of the boundary~$\partial\Omega$
(with the convention that $\kappa<0$ if~$\Omega$ is convex)
and $\gamma: W^{1,2}(\Omega;\C^2) \to \sii(\partial\Omega;\C^2)$ 
is the boundary-trace operator. 
Formally, this is easily seen by expanding 
$\|(\partial_1 + i\partial_2)u_1\|^2$
and $\|(\partial_1 - i\partial_2)u_2\|^2$
and integrating by parts. 
To justify this approach,
one needs an extra regularity of~$u$.
This is certainly a non-trivial matter because, 
while the curvature is piece-wise zero for triangles,
it is not defined at the vertices.

Our main ingredient to prove 
an analogue of the useful formula~\eqref{curvature} 
for triangles is
the following density result.
The idea and proof is due to 
D.~Krej\v{c}i\v{r}\'ik \cite{private}.

\begin{Lemma}\label{Lem.polygon} 
Let $\Omega$ be any two-dimensional polygon 
with the set of vertices~$V$. 
Then 
$$
  \mathcal{D} := \Dom(T) \cap C_0^\infty(\Real^2\setminus \{ V \})
$$ 
is a core of~$T$.
\end{Lemma}
\begin{proof}
Clearly, it is enough to consider the massless case $m=0$.
Moreover, by partition of unity, it suffices to consider
the sector 
$
  \Omega := \{(r\cos\theta,r\sin\theta): \, 
  r \in (0,\infty), \ \theta \in (0,\alpha)
  \}
$
with $\alpha \in (0,2\pi)$. 
Let us denote
$
  \partial_\theta \Omega := \{(r\cos\theta,r\sin\theta): \, 
  r \in (0,\infty)
  \}
$.
Consider the Dirac operator~\eqref{operator1}--\eqref{operator2}
(with $m=0$),
which involves the infinite-mass boundary conditions
\begin{equation}\label{bc}
\begin{aligned}
  u_2 &=  u_1 && \mbox{on} \quad \partial_0\Omega \,,
  \\
  u_2 &= e^{-i\alpha }u_1 && \mbox{on} \quad \partial_\alpha\Omega  \,,
\end{aligned}  
\end{equation}
in the sense of traces in 
$
  W^{1,2}(\Omega;\C^2)
  \ni u = 
  \begin{psmallmatrix}
    u_1 \\ u_2 
  \end{psmallmatrix}
$. 
More specifically,
$
  \Dom(T) = \left\{ 
  u \in W^{1,2}(\Omega;\C^2) : \
  \eqref{bc} \ \mbox{holds} 
  \right\}
$.

The crucial observation is that the functions
\begin{equation*}
\begin{aligned}
  \psi_1 &:= u_2 - u_1 \,, 
  \\
  \psi_2 &:= u_2 - e^{-i\alpha} u_1 \,, 
\end{aligned}  
\end{equation*}
satisfy the Dirichlet boundary condition 
on $\partial_\beta\Omega$ and $\partial_\alpha\Omega$, respectively. 
The inverse formulae are given by 
\begin{equation*}
\begin{aligned}
  u_1 &:= \frac{1}{e^{-i\alpha}- 1} \, 
  (\psi_1 - \psi_2)\,, 
  \\
  u_2 &:= \frac{1}{e^{-i\alpha}- 1} \,
  (e^{-i\alpha} \psi_1 - \psi_2) \,.
\end{aligned}  
\end{equation*}

\subsubsection*{Step 1: Approximation by bounded functions}
For any function $\phi:\Omega \to \Real$,
define the vertical cut-off
\begin{equation*}
  \phi^N(x) :=
  \begin{cases}
    N & \mbox{if} \quad \phi(x) > N \,,
    \\
    \phi(x) & \mbox{if} \quad |\phi(x)| \leq N \,,
    \\
    -N & \mbox{if} \quad \phi(x) < -N \,.
  \end{cases} 
\end{equation*}
By definition,
$
  \|\phi^N\|_{L^\infty(\Omega)} \leq N
$.
If $\phi \in W^{1,2}(\Omega)$, we have \\
\begin{equation*}
\begin{aligned}
\|\phi -\phi^N \|^2_{W^{1,2}(\Omega)}& = \int_\Omega |\phi-\phi^N|^2 \, dx + \int_\Omega |\nabla\phi -\nabla\phi^N |^2 \, dx \\
&= \int_{\{|\phi| > N \}} |\phi-N|^2 \, dx + \int_{\{|\phi| > N \}} |\nabla\phi|^2 \, dx\\
&\leq \int_{\{|\phi| > N \}} 4 \, |\phi(x)|^2 \, dx + \int_{\{|\phi| > N \}} |\nabla \phi(x)|^2 \, dx 
\xrightarrow[N\to\infty]{} 0 \,.
\end{aligned}
\end{equation*}
Hence $\phi^N \to \phi$ in $W^{1,2}(\Omega)$ as $N \to \infty$.

In our case, if $u \in \Dom(T)$, we set
\begin{equation*}
  \psi^N:=
  \begin{pmatrix}
    (\Re \psi_1)^N + i (\Im \psi_1)^N \\
    (\Re \psi_2)^N + i (\Im \psi_2)^N 
  \end{pmatrix} 
  \qquad \mbox{and} \qquad
  u^N:=
  \frac{1}{e^{-i\alpha}-1} \, 
  \begin{pmatrix}
    \psi_1^N - \psi_2^N 
    \\
    e^{-i\alpha} \psi_1^N - \psi_2^N 
  \end{pmatrix} 
  .
\end{equation*}
Then $u^N \in \Dom(T) \cap L^\infty(\Omega;\C^2)$
and \\
\begin{equation*}
\begin{aligned}
\|u^N - u \|^2_{W^{1,2}(\Omega;\C^2)} 
&= \frac{1}{|e^{-i\alpha} -1|^2}  \|\psi_1 - \psi_1^N + \psi_2 -\psi_2^N \|^2_{W^{1,2}(\Omega)} \\
& \quad + \frac{1}{|e^{-i\alpha} -1|^2} \|e^{-i\alpha}(\psi_1 - \psi_1^N) + \psi_2 -\psi_2^N \|^2_{W^{1,2}(\Omega)} \\
&\leq \frac{2}{|e^{-i\alpha} -1|^2} ( \|\psi_1 - \psi_1^N\|^2_{W^{1,2}(\Omega)} + \|\psi_2 -\psi_2^N \|^2_{W^{1,2}(\Omega)} ) 
\xrightarrow[N\to\infty]{}
 0 \,.
\end{aligned}
\end{equation*}
Therefore, $u^N \to u$ in $W^{1,2}(\Omega;\C^2)$ as $N \to \infty$.
Consequently, 
\begin{equation*}
\mathcal{D}_1 := \Dom(T) \cap L^\infty(\Omega;\C^2)
\mbox{ is a core of~$T$.}
\end{equation*}

\subsubsection*{Step 2: Approximation by compactly supported functions}
Consider the cut-off sequence 
$\xi_n:[0,\infty) \to [0,1]$
defined for every $n \geq 2$ by
\begin{equation*}
  \xi_n(r) :=
  \begin{cases}
    0 
    & \mbox{if} \quad r \in [0,n^{-2}) \,, 
    \\
    \displaystyle
    \frac{\log(n^2 r)}{\log(n)}  
    & \mbox{if} \quad r \in [n^{-2},n^{-1}) \,, 
    \\
    1
    & \mbox{if} \quad r \in [n^{-1},n) \,, 
    \\
    \displaystyle
    \frac{\log(n^{-2} r)}{\log(n^{-1})}  
    & \mbox{if} \quad r \in [n,n^{2}) \,, 
    \\
    0 
    & \mbox{if} \quad r \in [n^{2},\infty) \,.
  \end{cases} 
\end{equation*}
For every $u \in \mathcal{D}_1$,
define $u_n(x) := \xi_n(|x|) u(x)$.
Clearly, $u_n \in L^\infty_0(\Omega;\C^2)$,
by which we mean that~$u_n$ is bounded and~vanishes 
in a neighbourhood of~$0$
as well as in a neighbourhood of infinity.
Moreover, $u_n \in \Dom(H)$.

Since $\xi_n \to 1$ pointwise as $n \to \infty$,
it is easy to see that $u_n \to u$ in $\sii(\Omega;\C^2)$
as $n \to \infty$ by the dominated convergence theorem. 
Writing
\begin{equation*}
  \int_{\Omega \cap \{|x| \geq 1\}} |\nabla (u_n-u)|^2
  \leq  2\int_{\Omega \cap \{|x| \geq 1\}} (\xi_n-1)^2 \, |\nabla u|^2
  + 2 \int_{\Omega \cap \{|x| \geq 1\}} |\nabla\xi_n|^2 \, |u|^2
  \,,
\end{equation*}
we see that the first term on the right-hand side 
tends to zero as $n \to \infty$,
as above due to the dominated convergence theorem. 
For the second term, we estimate
\begin{equation*}
  \int_{\Omega \cap \{|x| \geq 1\}} |\nabla\xi_n|^2 \, |u|^2
  \leq \|u\|_{L^\infty(\Omega;\C^2)}^2
  \int_{\Omega \cap \{|x| \geq 1\}} |\nabla\xi_n|^2 
\end{equation*}
and use the polar coordinates to control 
the last integral as follows:
\begin{equation*}
  \int_{\Omega \cap \{|x| \geq 1\}} |\nabla\xi_n|^2 
  = 2\pi \int_n^{n^2} \frac{1}{r^2 \log^2(n)} \, r \, \der r
  = \frac{2\pi}{\log(n)}
  \xrightarrow[n \to \infty]{} 0 \,.
\end{equation*}
In a similar manner, we verify that 
\begin{equation*}
  \int_{\Omega \cap \{|x| \leq 1\}} |\nabla (u_n-u)|^2
  \xrightarrow[n \to \infty]{} 0 \,.
\end{equation*}
Consequently, 
\begin{equation*}
\mathcal{D}_2 := \Dom(T) \cap L_0^\infty(\Omega;\C^2)
\mbox{ is a core of~$T$.}
\end{equation*}

\subsubsection*{Step 3: Approximation by smooth functions}
Let $u \in \mathcal{D}_2$. 
Then the function $\psi := u_2 - e^{-i\theta} u_1$
is well defined, where~$\theta(x)$ is the unique number
in $[0,\alpha]$ with 
$x_1 = |x| \cos[\theta(x)]$ and $x_2 = |x| \sin[\theta(x)]$ for every $x \in \Omega$.
Since~$\psi$ vanishes on~$\partial\Omega$,
there exists a sequence $\{\psi^j\} \subset C_0^\infty(\Omega)$
such that $\psi^j \to \psi$ in $W^{1,2}(\Omega;\C^2)$
as $j \to \infty$.  

Since~$\Omega$ satisfies the segment condition,
there also exists a sequence 
$\{u_1^j\} \subset C_0^\infty(\Real^2)$ 
such that $u_1^j \to u_1$ in $W^{1,2}(\Omega;\C^2)$
as $j \to \infty$. 
Since~$u_1$ vanishes in a neighbourhood of zero,
the sequence can be chosen to lie in 
$C_0^\infty(\Real^2 \setminus \{0\})$.  

Define 
$
  u_2^j := \psi^j + e^{-i\theta} u_1^j 
  \in C_0^\infty(\Real^2 \setminus \{0\})
$.
Then $u_2^j \to u_2$ in $W^{1,2}(\Omega;\C^2)$
as $j \to \infty$. 
Moreover, $u^j$~satisfies~\eqref{bc}. 
Consequently, 
\begin{equation}
\mathcal{D} = \Dom(T) \cap L_0^\infty(\Omega;\C^2)
\cap C_0^\infty(\Real^2\setminus \{0\})
\mbox{ is a core of~$T$.}
\end{equation}
This concludes the proof of the lemma.
\end{proof}

As a special consequence, 
the norm of~$T_{a,b} u$ can be computed explicitly 
by using integration by parts.
Formally, the result coincides 
with the formula~\eqref{curvature} for smooth domains
with $\kappa=0$.
 
\begin{Theorem}\label{formula}
For every $u \in \Dom(T_{a,b})$,
\begin{equation}\label{norm}   
  \|T_{a,b} u\|^2   = \|\nabla u\|^2_{\sii(\Omega_{a,b})} 
  + m^2 \|u\|^2_{\sii(\Omega_{a,b})} + m \, \|\gamma u\|^2_{\sii ( \partial\Omega_{a,b}) } .
\end{equation}
\end{Theorem}
\begin{proof}
By virtue of Lemma \eqref{Lem.polygon}, 
for every $u = \begin{psmallmatrix}
  u_1 \\ u_2 
  \end{psmallmatrix}
  $
   in $\Dom(H_{a,b})$, there exists a sequence\\
    $\{u_n = \begin{psmallmatrix}
  u_{1n} \\ u_{2n} 
  \end{psmallmatrix}
\} $ in $ \mathcal{D}$ such that 
$ u_n \xrightarrow[n \to \infty]{} u$
in $H^1(\Omega_{a,b})$. 
Using integration by parts, we compute:
%
\begin{equation*}
\begin{aligned}
\lefteqn{\|T_{a,b} u_n \|^2}   
\\
&= \|m u_{1n} -i (\partial_1- i\partial_2) u_{2n}\|^2 + \|m u_{2n} + i (\partial_1 + i \partial_2 ) u_{1n}\|^2 \\
& = m^2 \|u_n \|^2 + \|\nabla u_n\|^2 - 2 m \Re ( u_{1n}, i(\partial_1- i\partial_2) u_{2n}) \\
& \quad + 2m \Re( u_{2n}, i (\partial_1 + i \partial_2 ) u_{1n}) - 2\Re (\partial_1 u_{2n} ,i \partial_2 u_{2n}) + 2\Re( \partial_1 u_{1n}, i \partial_2 u_{1n}) \\
& = m^2 \|u_n \|^2 + \|\nabla u_n\|^2 +  2m \Re \int_{0}^{b} i {\bar{u}_{1n}}(0,x_2) u_{2n}( 0,x_2) \, dx_2  \\
 & \quad - 2m \Re\int_{0}^{b} i {\bar{u}_{1n}}(-\frac{a}{b}x_2 +a,x_2) u_{2n}( -\frac{a}{b}x_2 +a,x_2) \, dx_2  \\ 
 & \quad + 2m \Re \int_{0}^{a}  {\bar{u}_{1n}}(x_1,0) u_{2n}( x_1,0) \, dx_1 \\
& \quad -2m \Re \int_{0}^{a}  {\bar{u}_{1n}}(x_1, -\frac{b}{a}x_1 +b) u_{2n}( x_1,-\frac{b}{a}x_1 +b) \, dx_1  \\
& \quad - 2 \Re \int_{0}^{b} i {\bar{u}_{1n}}(0,x_2) \partial_2 u_{1n}( 0,x_2) \, dx_2 + 
 2 \Re \int_0^b i {\bar{u}_{1n}}(-\frac{a}{b} x_2 + a,x_2) \partial_2 u_{1n}( -\frac{a}{b} x_2 + a,x_2) \, dx_2  \\
& \quad + 2 \Re \int_{0}^{a} i {\bar{u}_{1n}}(x_1,0) \partial_2 u_{1n}( x_1,0) \, dx_1 - 2 \Re \int_0^a i {\bar{u}_{1n}}(-\frac{b}{a} x_1 + b,x_1) \partial_2 u_{1n}( x_1,-\frac{b}{a} x_1 + b) \, dx_1   \\
& \quad +
2 \Re \int_{0}^{b} i {\bar{u}_{2n}}(0,x_2) \partial_2 u_{2n}( 0,x_2) \, dx_2 
 - 2 \Re \int_0^b i {\bar{u}_{2n}}(-\frac{a}{b} x_2 + a,x_2) \partial_2 u_{2n}( -\frac{a}{b} x_2 + a,x_2) \, dx_2  \\
& \quad - 2 \Re \int_{0}^{a} i {\bar{u}_{2n}}(x_1,0) \partial_2 u_{2n}( x_1,0) \, dx_1 +
   2 \Re \int_0^a i {\bar{u}_{2n}}(-\frac{b}{a} x_1 + b,x_1) \partial_2 u_{2n}( x_1,-\frac{b}{a} x_1 + b) \, dx_1  . 
\end{aligned}
\end{equation*}
Substituting the boundary conditions, we have 
\begin{equation*}
\begin{aligned}
\int_{0}^{b} i {\bar{u}_{1n}}(0,x_2) u_{2n}( 0,x_2) dx_2 
&= \int_{0}^{b}  |u_{1n}(0,x_2)|^2 dx_2 = \|u_{1n}\|^2_{OB} = \frac{1}{2}\|u_{n}\|^2_{OB}   \, , \\
\int_{0}^{a}  {\bar{u}_{1n}}(x_1,0) u_{2n}( x_1,0) dx_1 
&= \int_{0}^{b}  |u_{1n}(x_1,0)|^2 dx_1 = \|u_{1n}\|^2_{OA}  = \frac{1}{2} \|u_{n}\|^2_{OA} \, , \\
\Re\int_{0}^{b} i {\bar{u}_{1n}}(-\frac{a}{b}x_2 +a,x_2) u_{2n}( -\frac{a}{b}x_2 +a,x_2) dx_2 
&= \frac{-b}{\sqrt{a^2+b^2}}\int_{0}^{b}  |u_{1n}(-\frac{a}{b}x_2 +a,x_2)|^2 dx_2 \\
&= \frac{-b^2}{a^2+b^2}\|u_{1n}\|^2_{AB} = \frac{1}{2}\frac{-b^2}{a^2+b^2}\|u_{n}\|^2_{AB}   \, , \\
\Re\int_{0}^{a} i {\bar{u}_{1n}}(-\frac{b}{a}x_1 + b,x_1) u_{2n}( -\frac{b}{a}x_1 +b,x_1) dx_1 
&=  \frac{-a^2}{a^2+b^2}\|u_{1n}\|^2_{AB} = \frac{1}{2}\frac{-a^2}{a^2+b^2}\|u_{n}\|^2_{AB}  \, .
\end{aligned}
\end{equation*}
Moreover, using the boundary conditions,
an integration by parts on the edge of the triangle
and the fact that the approximating sequence vanishes 
in a vicinity of the vertices, we have
\begin{equation*}
\begin{aligned}
\lefteqn{
 \Re \int_{0}^{b} i {\bar{u}_{2n}}(0,x_2) \partial_2 u_{2n}( 0,x_2) dx_2}
 \\
  &= \Re \int_{0}^{b} - {\bar{u}_{1n}}(0,x_2) \partial_2 u_{2n}( 0,x_2) dx_2 \\
 & = \Re \int_{0}^{b} \partial_2 {\bar{u}_{1n}}(0,x_2)  u_{2n}( 0,x_2) dx_2 + {\bar{u}_{1n}}(0,0)  u_{2n}( 0,0) - {\bar{u}_{1n}}(0,b)  u_{2n}( 0,b) \\
 & = \Re \int_{0}^{b} \partial_2 {\bar{ u}_{1n}}(0,x_2)  u_{2n}( 0,x_2) dx_2 
 \\
 &=  \Re \int_{0}^{b} -i \partial_2 {\bar{ u}_{1n}}(0,x_2)   u_{1n}( 0,x_2) dx_2  
 \\
& = \Re \int_0^b i \bar{u}_{1n}(0,x_2) \partial_2 u_{1n}(0,x_2) dx_2  \, .
\end{aligned}
\end{equation*}
By analogous manipulations, 
we have 
\begin{equation*}
\begin{aligned}
\Re \int_0^b i {\bar{u}_{1n}}(-\frac{a}{b} x_2 + a,x_2)& \partial_2 u_{1n}( -\frac{a}{b} x_2 + a,x_2) \, dx_2\\
& = \Re \int_0^b i {\bar{u}_{2n}}(-\frac{a}{b} x_2 + a,x_2) \partial_2 u_{2n}( -\frac{a}{b} x_2 + a,x_2) \, dx_2 \, , \\
\Re \int_{0}^{a} i {\bar{u}_{1n}}(x_1,0) &\partial_2 u_{1n}( x_1,0) \, dx_1 \\
&= \Re \int_{0}^{a} i {\bar{u}_{2n}}(x_1,0) \partial_2 u_{2n}( x_1,0) \, dx_1 \, , \\
\Re \int_0^a i {\bar{u}_{1n}}(-\frac{b}{a} x_1 + b,x_1)& \partial_2 u_{1n}( x_1,-\frac{b}{a} x_1 + b) \, dx_1 \\
&= \Re \int_0^a i {\bar{u}_{2n}}(-\frac{b}{a} x_1 + b,x_1) \partial_2 u_{2n}( x_1,-\frac{b}{a} x_1 + b) \, dx_1  \, ,\\ 
\end{aligned}
\end{equation*}
Putting all these identities together,
we obtain the formula
\begin{equation*}   
  \|T_{a,b} u_n\|^2   = \|\nabla u_n\|^2_{\sii(\Omega_{a,b})} 
  + m^2 \|u_n\|^2_{\sii(\Omega_{a,b})} + m \, \|\gamma u_n\|^2_{\sii ( \partial\Omega_{a,b}) }
  \end{equation*}
  valid for all $u_n \in \mathcal{D}$. Taking
 $n \rightarrow \infty$, we obtain the desired result.
\end{proof}
\begin{Remark}
Applying the similar arguments, 
we can prove the validity of an analogue of
the formula~\eqref{norm} for arbitrary planar polygons 
(see Appendix~\ref{A}).
\end{Remark}
%

 \section{One-dimensional Dirac operator\txtD{s}}\label{Sec.well}
%
For further purposes, 
given any $m \geq 0$ and arbitrary
positive numbers $a,b$ and~$L$,
let us consider the one-dimensional Dirac operator 
\begin{equation}\label{operator.1D}
\begin{aligned}
  H_{Lm}	 &:= 
  \begin{pmatrix}
    -i \partial & -m \\
    -m & i \partial 
  \end{pmatrix}
  \qquad \mbox{in} \qquad
  \sii\left(\left(0,L \right);\C^2\right)
  \,,
  \\
  \Dom(H_{Lm}) &:= \left\{
  \varphi \in W^{1,2}\left(\left(0,L \right);\C^2\right), \
  \varphi_2(L) =  \left( \frac{-a}{\sqrt{a^2+b^2}}+ i \, \frac{b}{\sqrt{a^2+b^2}}\right)  \varphi_1(L) \, , \varphi_2(0) =  \varphi_1(0)
  \right\}
  \,.
\end{aligned}  
\end{equation}
\begin{Proposition}
The operator $H_{Lm}$ is self-adjoint. 
\end{Proposition}
\begin{proof}
We follow \cite[App.~A]{BBKO}.
Since the multiplication by 
$
\begin{psmallmatrix}
    0 & -m \\
    -m & 0 
  \end{psmallmatrix}
$
generates a bounded self-adjoint operator 
on $L^2((0,L); \mathbb{C}^2)$, 
it is enough to prove the self-adjointness of~$H_{L0}$.
To do that, we commence with 
the definition of the adjoint
\begin{multline*}
		\Dom(H_{L0}^*) = \Big\{ u \in L^2\big((0,L);\C^2\big) :\exists w \in L^2\big((0,L);\C^2\big) \text{ such that }   \forall  v \in \Dom(H_{L0}), \\ \ (  u, H_{L0} v)_{L^2((0,L);\C^2)} = (  w,v)_{L^2((0,L);\C^2)}\Big\}.
\end{multline*}
For every
 $v \in C_0^\infty\big((0,L);\C^2\big)$ 
 and $u \in \Dom(H_{L0}^*)$, there holds

\begin{align*}
	(  H_{L0}^*u,v)_{L^2((0,L);\C^2)} = (  u,H_{L0} v)_{L^2((0,L);\C^2)} 
	&= ( u_1, -i v_1' )_{L^2((0,L);\C)} + ( u_2, i v_2' )_{L^2((0,L);\C)}\\
	&= \left\langle \begin{psmallmatrix}
   u_1 \\ u_2
  \end{psmallmatrix}, \begin{psmallmatrix}
  -i v_1' \\ i v_2' 
  \end{psmallmatrix} \right\rangle_{\cD',\cD}\\
	& = \left\langle \begin{psmallmatrix}
  -i u_1' \\ i u_2' 
  \end{psmallmatrix}, \begin{psmallmatrix}
  v_1 \\ v_2 
  \end{psmallmatrix} \right\rangle_{\cD',\cD} ,
\end{align*}
where $\langle\cdot,\cdot\rangle_{\cD',\cD}$ is the duality bracket of distributions. In particular, 
we know that 
$H_{L0}^*u = \begin{psmallmatrix}
  -i u_1' \\ i u_2' 
  \end{psmallmatrix} \in L^2\big((0,L);\C^2\big)$, 
  thus we get $u\in H^1\big((0,L);\C^2\big)$. Moreover, if $v \in \Dom (H_{L0}) $ there holds
\begin{align*}
	(  H_{L0}^* u,v)_{L^2((0,L);\C^2)} &= ( u_1, -i v_1' )_{L^2((0,L);\C)} + ( u_2, i v_2' )_{L^2((0,L);\C)}\\
	&= ( -i u_1', v_1 )_{L^2((0,L);\C)} + ( i u_2,  v_2' )_{L^2((0,L);\C)} + i \overline{u_2} v_2 |_0^L -i \overline{u_1} v_1 |_0^L\\
	&= ( -i u_1', v_1 )_{L^2((0,L);\C)} + ( i u_2,  v_2' )_{L^2((0,L);\C)} \\
	& \quad + i v_1(L)[ \overline{u_2}(L) \left( \frac{-a+ b i}{\sqrt{a^2+b^2}}\right) - \overline{u_1(L)}] + i v_1(0) [\overline{u_1(0)} - \overline{u_2(0)}] .
\end{align*}
Since it holds for any arbitrary $v\in\Dom(H_{L0})$ we obtain that $u\in \Dom(H_{L0})$, that is $\Dom(H_{L0})=\Dom(H_{L0}^*)$. 
\end{proof}

By dint of the compactness embedding 
$H^1((0,L);\C^2)$ into $L^2((0,L);\C^2)$
and since $\Dom(H_{Lm})$ is continuously embedded in
 $H^1((0,L);\C^2)$ then we deduce that the spectrum $  \Sp(H_{Lm})$ of the self-adjoint operator $H_{Lm}$ is purely discrete.
In the following, we compute the eigenvalues. 
 
First of all, we observe that 
any eigenvalue $\lambda \in \Sp(H_{Lm}) $ 
necessarily satisfies $ \lambda^2 > m^2$.
Indeed, by computing the square norm of the operator
 \begin{equation}\label{1Dnorm}
 \begin{aligned}
 \|H_{Lm} u\|^2 &= \|i u_1' + m u_2\|^2 + \|i u_2' -m u_1\|^2\\
 &= \|u'\|^2 + 2\Re m (i u_1',u_2)- 2\Re m (u_1,i u_2') + m^2 \|u\|^2\\
& = \|u'\|^2 + m^2 \|u\|^2 + 2\Re m \int_0^L \overline{i u_1'} u_2 - 2\Re m \int_0^L \overline{u_1} i u_2'\\
 &= \|u'\|^2 + m^2 \|u\|^2 - 2\Re m i \overline{u_1}u_2|_0^L\\
 &= \|u'\|^2 + m^2 \|u\|^2 + m \frac{b}{\sqrt{a^2+b^2}} |u(L)|^2 
 \geq m^2 \|u\|^2,
 \end{aligned}
 \end{equation}
we immediately obtain $ \lambda^2 \geq m^2$.
The inequality is actually strict because $ \lambda^2 = m^2$
would imply that~$u$ is a constant, 
which is impossible unless $u=0$.
 
Let $\lambda \in \Sp(H_{Lm})$ and 
let $u = (u_1,u_2)^\top \in \Dom(H_{Lm})$ 
be an associated eigenfunction. 
It satisfies 
\begin{equation}\label{sp}
\left\{\begin{array}{rcl}
	-i u_1' -m u_2 &=& \lambda u_1\,,\\
	i u_2' - m u_1&= & \lambda u_2\,,
\end{array}\right.
\end{equation}
or equivalently
\begin{equation}\label{sp1}
\left\{\begin{array}{rcl}
	 -m u_2 &=&i u_1' + \lambda u_1\,,\\
	 m u_1&= & i u_2' -\lambda u_2\,.
\end{array}\right.
\end{equation}
Differentiating both sides of the equations~\eqref{sp} 
and combining with the equations~\eqref{sp1}, 
we obtain differential equations 
that the components of~$u$ must satisfy separately
$$
-u''_1= (\lambda^2 - m^2) u_1\,,\qquad - u_2''= (\lambda^2- m^2) u_2\,.
$$
Putting $ E := \lambda^2 - m^2 > 0$,
the general solutions read
%
\begin{equation}\label{eqs}
\left\{\begin{array}{rcl}
	u_1(x) &=& C_1 \cos \sqrt{E} x + C_2 \sin \sqrt{E} x \,,\\
	 u_2(x)&= & \tilde{C_1} \cos\sqrt{E} x + \tilde{C_2} \sin\sqrt{E} x\,,
\end{array}\right.
\end{equation}
where $C_1,C_2,\tilde{C_1},\tilde{C_2}$ are complex constants.
The boundary condition $u_2(0)= u_1(0)$
directly implies that $C_1= \tilde{C_1}$.
Substituting this expression of~$u$ 
into \eqref{sp}, we have
%
\begin{equation*}
\begin{aligned}
-i C_1 \sqrt{E}\sin\sqrt{E}x + i C_2 \sqrt{E} \cos\sqrt{E} x + m C_1 \cos\sqrt{E} x + m \tilde{C_2} \sin\sqrt{E} x 
&= - \lambda ( C_1 \cos\sqrt{E} x + C_2 \sin\sqrt{E} x ) ,
\\
-i C_1 \sqrt{E}\sin\sqrt{E}x + i \tilde{C_2} \sqrt{E} \cos\sqrt{E} x - m C_1 \cos\sqrt{E} x - m C_2 \sin\sqrt{E} x 
&=  \lambda ( C_1 \cos\sqrt{E} x + \tilde{C_2} \sin\sqrt{E} x ) .
\end{aligned}
\end{equation*}
%
From these equalities, we deduce
\begin{equation*}
\left\{\begin{array}{rcl}
	C_2 &=& \frac{i(\lambda + m)}{\sqrt{E}} C_1  \,,\\
	 \tilde{C_2}  &= & \frac{i  (-\lambda-m)}{\sqrt{E}} C_1 
\end{array}\right.
\end{equation*}
Putting $ M := \frac{\lambda+m}{\sqrt{E}} \neq 0 $, 
then we obtain

\begin{equation}\label{eqs1}
\left\{\begin{array}{rcl}
	u_1(x) &=& C_1 \cos\sqrt{E} x + i M C_1 \sin\sqrt{E} x \,,\\
	 u_2(x) &= & C_1 \cos\sqrt{E} x - i M C_1 \sin\sqrt{E} x\,,
\end{array}\right.
\end{equation}
with $C_1$ being a non-zero complex constant.
The boundary condition 
$u_2(L)= \left( \frac{-a+ b i}{\sqrt{a^2+b^2}}\right) u_1(L)$ requires
\begin{equation*}
\frac{C_1 \cos\sqrt{E} L - i M C_1 \sin\sqrt{E} L}{C_1 \cos\sqrt{E} L + i M C_1 \sin\sqrt{E} L}=  \frac{-a+ b i}{\sqrt{a^2+b^2}} 
\end{equation*}
which is equivalent to
%
\begin{equation*}
 \frac{\cos^2\sqrt{E} L - M^2 \sin^2\sqrt{E} L - i M \sin2\sqrt{E} L}{\cos^2\sqrt{E} L + M^2 \sin^2\sqrt{E} L} = \frac{-a+ b i}{\sqrt{a^2+b^2}} \,.
\end{equation*}
Considering the real and imaginary parts separately,
it is equivalent to the system
\begin{equation*}
\left\{\begin{array}{rcl}
	\frac{\cos^2\sqrt{E} L - M^2 \sin^2\sqrt{E} L}{\cos^2\sqrt{E} L + M^2 \sin^2\sqrt{E} L} &=& \frac{-a}{\sqrt{a^2+b^2}} \, ,\\
\frac{M \sin2\sqrt{E} L}{\cos^2\sqrt{E} L + M^2 \sin^2\sqrt{E} L}	  &= & \frac{-b}{\sqrt{a^2+b^2}}\, .
\end{array}\right.
\end{equation*}
From the second equation we infer that $ M \sin2\sqrt{E} L < 0$,
so $\cos\sqrt{E}L$ can not be equal zero.
 Putting $z := M \tan\sqrt{E} L < 0 $ 
 and dividing both the numerator and denominator of the left fractions of these equations by $\cos^2\sqrt{E} L$,
we have 
\begin{equation*}
\left\{\begin{array}{rcl}
	\frac{1- z^2}{1+ z^2} &=& \frac{-a}{\sqrt{a^2 + b^2}}  \,,\\
    \frac{2 z}{ 1+ z^2}  &= & \frac{-b}{\sqrt{a^2+b^2}} \,. 
\end{array}\right.
\end{equation*}
Therefore, $z$ satisfies the equation 
$$ \frac{1-z^2}{2 z} = \frac{a}{b} \,.$$
Since $z$ is the negative solution of the quadratic equation 
$ z^2 + 2 \frac{a}{b} z -1 =0$, one has  
 $$ 
 z = z_0 := - \frac{a}{b} - \sqrt{\frac{a^2}{b^2}+1} < 0  .
 $$ 
 In summary, the eigenvalue~$\lambda$ of $H_{Lm}$ satisfies
 the implicit equation
 \begin{equation}\label{implicit1}
 \frac{\lambda+m}{\sqrt{\lambda^2-m^2}} \tan\sqrt{\lambda^2 - m^2} L -z_0 =0.
 \end{equation}

From the formula \eqref{1Dnorm} we have 
 $$\|H_{Lm} u\|^2 - m^2 \|u\|^2 \geq \|u'\|^2 = \|H_{L0}u\|^2 $$ for all $u \in \Dom(H_{Lm}) = \Dom(H_{L0})$.  
Let $\lambda_1 = \lambda_1(m)$ be the closest-to-zero 
eigenvalue of $H_{Lm}$ and set
$ E_1(m) := \lambda_1(m)^2-m^2$.
  As a consequence,  
  $$
  E_1(m)=\lambda_1(m)^{2}  - m^2 \geq \lambda_1(0)^2 = E_1(0)
  $$
for every $m \geq 0$. 
 When $ m=0$ we have 
%
\begin{equation*}
\begin{aligned}
&- \frac{a}{b} - \sqrt{\frac{a^2}{b^2}+1} = - \tan\sqrt{E_1(0)} L 
\\
&\Longleftrightarrow \tan\sqrt{E_1(0)} L = \frac{a}{b} + \sqrt{\frac{a^2}{b^2}+1} > 1 \\
&\Longrightarrow \sqrt{E_1(0)} L = \arctan(-z_0) > \frac{\pi}{4} 
\\
&\Longrightarrow \lambda_1^{2} (m)  - m^2 \geq \lambda_1^2 (0) = \frac{\arctan|z_0|^2}{L^2} > \frac{\pi^2}{16 L^2} \, .
\end{aligned}
\end{equation*} 
Applying a variational formulation analogous to~\eqref{eq2}
and a unitary equivalence,
we have just established the following
Poincar\'e-type inequality. 
\begin{Lemma}\label{Poincare}
For every $\phi \in \Dom(H_{Lm})$,
\begin{equation*}
 \|\phi'\|^2 + m \frac{b}{\sqrt{a^2+b^2}} |\phi(L)|^2  \geq (\lambda_1^2 - m^2 ) \|\phi\|^2 \geq \frac{\arctan^2(\frac{a}{b} + \sqrt{\frac{a^2}{b^2}+1})}{L^2}   \|\phi\|^2 \,.
\end{equation*}
\end{Lemma}

Simultaneously with $H_{Lm}$,
we also consider the operator $G_{Lm}$
which acts as~$H_{Lm}$ but has a different domain:
$$ \Dom(G_{Lm}) := \left\{
  \varphi \in W^{1,2}\left(\left(0,L \right);\C^2\right), \
  \varphi_2(L) =  \left( \frac{-a}{\sqrt{a^2+b^2}}+ i \, \frac{b}{\sqrt{a^2+b^2}}\right)  \varphi_1(L) \, , \varphi_2(0) =  -i \varphi_1(0)
  \right\} 
  \,.$$
By the similar approach, we have $G_{Lm}$ is self-adjoint and its spectrum is purely discrete. We compute the square norm of the operator
\begin{equation}\label{1Dnorm2}
 \begin{aligned}
 \|G_{Lm} u\|^2 &= \|i u_1' + m u_2\|^2 + \|i u_2' -m u_1\|^2\\
 &= \|u'\|^2 + 2\Re m (i u_1',u_2)- 2\Re m (u_1,i u_2') + m^2 \|u\|^2\\
& = \|u'\|^2 + m^2 \|u\|^2 + 2\Re m \int_0^L \overline{i u_1'} u_2 - 2\Re m \int_0^L \overline{u_1} i u_2'\\
 &= \|u'\|^2 + m^2 \|u\|^2 - 2\Re m i \, \overline{u_1}u_2|_0^L\\
 &= \|u'\|^2 + m^2 \|u\|^2 + m \frac{b}{\sqrt{a^2+b^2}} |u(L)|^2  + m |u(0)|^2 \geq m^2 \|u\|^2.
 \end{aligned}
 \end{equation}
Thus if $\eta \in \Sp(G_{Lm})$ then $|\eta | > m $ for every $m \geq 0$ . Let $v = (v_1,v_2)^\top \in \Dom(G_{Lm})$ be an associated eigenfunction.
Then
\begin{equation}\label{sp2}
\left\{\begin{array}{rcl}
	-i v_1' -m v_2 &=& \eta v_1\,,\\
	i v_2' - m v_1&= & \eta v_2\,,
\end{array}\right.
\end{equation}
or equivalently
\begin{equation}\label{sp3}
\left\{\begin{array}{rcl}
	 -m v_2 &=&i v_1' + \eta v_1\,,\\
	 m v_1&= & i v_2' -\eta v_2\,.
\end{array}\right.
\end{equation}
Differentiating both sides of the equations \eqref{sp2} 
and combining with the equations \eqref{sp3}, we have 
$$
-v''_1= (\eta^2 - m^2) v_1\,,\qquad - v_2''= (\eta^2- m^2) v_2\,.
$$
Putting $ F := \eta^2 - m^2 > 0$, 
the general solutions read
\begin{equation}\label{eqs}
\left\{\begin{array}{rcl}
	v_1(x) 
	&=& D_1 \cos \sqrt{F} x + D_2 \sin \sqrt{F} x \,,\\
	 v_2(x) 
	 &= & \tilde{D_1} \cos\sqrt{F} x + \tilde{D_2} \sin\sqrt{F} x\,.
\end{array}\right.
\end{equation}
The boundary conditions
$u_2(0)= -i u_1(0)$
directly implies that $\tilde{D_1} = -i D_1 $.
Substituting this expression of~$u$ 
into \eqref{sp2}, we have
\begin{equation*}
\begin{aligned}
-i D_1 \sqrt{F}\sin\sqrt{F}x + i D_2 \sqrt{F} \cos\sqrt{F} x + m (-i D_1 \cos\sqrt{F} x +  \tilde{D_2} \sin\sqrt{F} x ) 
&= - \eta ( D_1 \cos\sqrt{F} x + D_2 \sin\sqrt{F} x ) , \\
- D_1 \sqrt{F}\sin\sqrt{F}x + i \tilde{D_2} \sqrt{F} \cos\sqrt{F} x - m D_1 \cos\sqrt{F} x - m D_2 \sin\sqrt{F} x 
&=  \eta ( -i D_1 \cos\sqrt{F} x + \tilde{D_2} \sin\sqrt{F} x ) .
\end{aligned}
\end{equation*}
From these equalities, we deduce
\begin{equation*}
\left\{\begin{array}{rcl}
	D_2 &=& \frac{(i \eta + m)}{\sqrt{F}} D_1  \,,\\
	 \tilde{D_2}  &= & \frac{-  (\eta + i m)}{\sqrt{F}} D_1 \,.
\end{array}\right.
\end{equation*}
Therefore 
\begin{equation}\label{eqs1}
\left\{\begin{array}{rcl}
	v_1(x) 
	&=& D_1 \cos\sqrt{F} x +  \frac{(i \eta + m)}{\sqrt{F}} D_1 \sin\sqrt{F} x \,,\\
	 v_2(x)
	 &= &-i D_1 \cos\sqrt{F} x  + \frac{-  (\eta + i m)}{\sqrt{F}} D_1 \sin\sqrt{F} x\,,
\end{array}\right.
\end{equation}
with $D_1$ being a non-zero complex constant.
From the boundary condition 
$u_2(L)= \left( \frac{-a+ b i}{\sqrt{a^2+b^2}}\right) u_1(L)$,
we deduce that $v_1(L) \neq 0 $ and
\begin{equation*}
\frac{ \cos\sqrt{F} L + \frac{(-i \eta + m)}{\sqrt{F}}  \sin\sqrt{F} L}{ \cos\sqrt{F} L + \frac{(i \eta + m)}{\sqrt{F}} \sin\sqrt{F} L}= i \frac{-a+ b i}{\sqrt{a^2+b^2}} \,,
\end{equation*}
which is equivalent to
\begin{equation*}
  \frac{\cos\sqrt{F} L +\frac{m}{\sqrt{F}} \sin\sqrt{F} L - \frac{i\eta}{\sqrt{F}}\sin \sqrt{F} L  }{\cos\sqrt{F} L + \frac{i\eta}{\sqrt{F}}  \sin\sqrt{F} L} = \frac{-a i- b }{\sqrt{a^2+b^2}} \,.
\end{equation*}
Putting $ A:= \cos\sqrt{F} L +\frac{m}{\sqrt{F}} \sin\sqrt{F} L $
and $ B:= \frac{\eta}{\sqrt{F}}\sin \sqrt{F} L $,
it is equivalent to the system
\begin{equation*}
\left\{\begin{array}{rcl}
	\frac{A^2- B^2}{A^2 + B^2} &=& \frac{-b}{\sqrt{a^2+b^2}} \, ,\\
\frac{2 A B}{A^2 + B^2}	  &= & \frac{a}{\sqrt{a^2+b^2}} > 0 \, .
\end{array}\right.
\end{equation*}
It implies that 
$ \frac{A}{B} = \frac{-b}{a} + \sqrt{\frac{b^2}{a^2} +1} $,
so
\begin{equation*}
 \frac{\sqrt{F}}{\eta}\cot \sqrt{F} L + \frac{m}{\eta} = \frac{-b}{a} + \sqrt{\frac{b^2}{a^2} +1} : = \alpha_0 > 0 .
 \end{equation*}
 or, equivalently, 
 \begin{equation}\label{sol1}
 f(\eta) :=
 \sqrt{F}\cot \sqrt{F} L + m - \alpha_0 \eta =0,
 \end{equation}
where~$f$ is defined on $(-\infty, -m) \cup (m, \infty)$. 
We compute the limits
$$
\lim_{\eta \rightarrow -m } f(\eta) = \frac{1}{L} + m + \alpha_0 m >0
 \qquad \mbox{and} \qquad
 \lim_{\eta \rightarrow -\sqrt{m^2 + \frac{\pi^2}{L^2}} } f(\eta) = -\infty. 
 $$ 
By the continuity of the left side of the equation \eqref{sol1}, we obtain that it always have solution when $ \sqrt{F} L \in (0, \pi)$ and thus, we restrict on this range to study 
properties of the closest-to-zero eigenvalue $\eta_1$. 

 If $\eta < 0$ then $\sqrt{F} \cot \sqrt{F} L + m <0$ and as a result, $ \cot\sqrt{F} L < 0$ and $  \sqrt{F} L \in ( \frac{\pi}{2}, \pi)  $. From the equation \eqref{sol1}, we have
 \begin{equation*}
 \begin{aligned}
&\sqrt{F} \cot\sqrt{F} L + m = \alpha_0 \eta \\
 &\Longrightarrow -\alpha_0 \eta = - \sqrt{F} \cot\sqrt{F} L -  m < \frac{F}{\pi - \sqrt{F} L} -m \\
& \Longrightarrow  \alpha_0 < \frac{-\eta}{m} \alpha_0 <  \frac{F}{(\pi - \sqrt{F} L)m} -1\\
 &\Longrightarrow m (\alpha_0 +1) < \frac{F}{\pi - \sqrt{F} L} \\
&\Longrightarrow \sqrt{F} > \frac{\pi}{(m \alpha_0 +m)^{-1} + L} \, .
 \end{aligned}
 \end{equation*}
If $\eta \geq 0$, from the implicit equation, we also obtain that
\begin{equation*}
\begin{aligned}
\eta'(m)& = \frac{m (\sin 2\sqrt{F} L  - 2 L \sqrt{F}) - 2 \sqrt{F} \sin^2\sqrt{F}L}{\eta (\sin2\sqrt{F} L - 2 L \sqrt{F})- 2 \alpha_0^2 \sqrt{F} \sin^2\sqrt{F}L}  > 0 \, , \\
(\eta^2 -m^2)'& = \frac{4 (\eta - m \alpha_0)\sqrt{F} \sin^2\sqrt{F} L}{\eta( -\sin 2\sqrt{E} L + 2 \sqrt{F} L) + 2 \alpha_0^2 \sqrt{F} \sin^2\sqrt{F}L} \, > 0.
\end{aligned}
\end{equation*}
%
It shows that $\eta^2 -m^2$ is a strictly increasing function with respect to $m$
but we can not give a m-dependent lower bound of $\eta_1^2 -m^2$ due to the fact that the spectrum of $G_{Lm}$ is not symmetric.

From the square norm of the operator \eqref{1Dnorm2}, we deduce that $ \eta_1^2 - m^2$ is a non-decreasing function and lies in the range $ (\frac{\pi^2}{16 L^2}, \frac{\pi^2}{L^2})$ and thus, 
$\lim_{m\rightarrow \infty} \frac{m}{|\eta_1|} =1$. 
We have $\{\eta_1^2 - m^2\}_{m\geq 0} $ is a non-decreasing 
sequence 
and being uniformly bounded then 
there exists $$\lim_{m\rightarrow \infty} \eta_1^2 - m^2 \leq \frac{\pi^2}{L^2}.
$$
Evaluating implicit equation \eqref{sol1} with $\lim_{m\rightarrow \infty} \frac{\sqrt{F}}{\eta_1} =0, \lim_{m\rightarrow \infty} \frac{m}{\eta_1}$ is finite,  $\sqrt{F}L \in (\frac{\pi}{4}, \pi)$. It implies that $\lim_{m\rightarrow \infty} \cot\sqrt{F}L $ must be equal  minus infinity and thus, we have 
$\sqrt{F}L \xrightarrow[m \to \infty]{} \pi$.
As a consequence, $\eta_1^2 - m^2 \rightarrow \frac{\pi^2}{L^2}$, 
which is the first eigenvalue of the Dirichlet Laplacian.
In summary, we have established the lower bound
$$ \lambda_1 ^2 (m) - m^2 \geq \lambda^2 (0)= \frac{\arctan^2(\frac{b}{a} +\sqrt{1+\frac{b^2}{a^2}})}{L^2} > \frac{\pi^2}{16 L^2} \, . $$
It yields the following Poincar\'e-type inequality.
\begin{Lemma} 
For every $\phi \in \Dom(G_{Lm})$, we have
\begin{equation*}
 \|\phi'\|^2 + m \frac{b}{\sqrt{a^2+b^2}} |\phi(L)|^2 + m |u(0)|^2  \geq (\lambda_1^2 - m^2 ) \|\phi\|^2 \geq \frac{\arctan^2(\frac{b}{a} +\sqrt{1+\frac{b^2}{a^2}})}{L^2} \|\phi\|^2 > \frac{\pi^2}{16 L^2} \|\phi\|^2 \,  .
\end{equation*}
\end{Lemma}
\begin{Remark}\label{1}
When $ \sqrt{F_1}L = \frac{\pi}{2} $ with respect to $ m = m_0$ then implicit equation \eqref{sol1} becomes 
$$ 0 + m_0 = \eta \alpha_0,$$
so $\eta_1^2 = \frac{m_0^2}{\alpha_0^2}$ and $ \eta_1^2 - m_0^2 = (\frac{1}{\alpha_0^2} -1 ) m_0^2 = \frac{\pi^2}{4L^2}$ . It implies that $m_0 = \frac{\pi}{2L} \sqrt{\frac{\alpha_0^2}{1- \alpha_0^2}} $.
For every $ m \leq m_0$ then  $ \sqrt{F_1}L \leq \frac{\pi}{2} $ and $ \cot\sqrt{F_1}L \geq 0$. 
Combining with equation \eqref{sol1}, we deduce that $ \eta_1 \alpha_0 -m \geq 0$ and thus we obtain the other lower bound $$F_1 = \eta_1^2 - m^2 \geq (\frac{1}{\alpha_0^2} -1 ) m^2 .$$
\end{Remark}
\begin{Remark}\label{2}
We can replace $H_{Lm}$ by 
the unitarily equivalent operator 
$\tilde{H}_{Lm} := v^*\tilde{G}_{Lm}v$ with 
$
  v :=
  \begin{psmallmatrix}
    -i & 0 \\
    0 & 1
  \end{psmallmatrix}
$,
where $\tilde{G}_{Lm}$ acts as $G_{Lm}$ 
but has a different domain:
$$ \Dom(\tilde{G}_{Lm}) := \left\{
  \varphi \in W^{1,2}\left(\left(0,L \right);\C^2\right), \
  \varphi_2(L) =  \left( \frac{a i}{\sqrt{a^2+b^2}}+  \, \frac{b}{\sqrt{a^2+b^2}}\right)  \varphi_1(L) \, , \varphi_2(0) =  -i \varphi_1(0)
  \right\} 
  \,.$$
  Thus $\Dom(\tilde{H}_{Lm}) = \Dom(H_{Lm})$ and  $\tilde{H}_{Lm}$ have the same spectrum  as $\tilde{G}_{Lm}$, 
  which is defined by the equation
  \begin{equation*}
 \sqrt{F}\cot \sqrt{F} L + m =  \eta (\frac{a}{b} + \sqrt{\frac{a^2}{b^2} +1}) = \eta |z_0| .
 \end{equation*}
\end{Remark}

 
\section{The spectral isoperimetric inequalities}\label{Sec.proof}
%

%

%

Now we are in a position to prove Theorem~\ref{Thm.bounds}.
\begin{proof}[Proof of Theorem~\ref{Thm.bounds}]
Recall the formula \eqref{eq2} 
that the first squared positive eigenvalue of the operator $T_{a,b}$ 
can be computed by

\begin{equation*}
  \lambda_1(a,b)^2 
  = \inf_{\stackrel[\psi \not= 0]{}{\psi \in \Dom(T_{a,b})}} 
  \frac{\|\hat{T}_{a,b}\psi\|^2}{\, \|\psi\|^2}  
\end{equation*}
with 
%
\begin{equation*}  
  \|T_{a,b} \psi\|^2   
  = \|\nabla \psi\|^2_{\sii(\Omega_{a,b})} 
  + m^2 \|\psi\|^2_{\sii(\Omega_{a,b})}
  + m \, \|\gamma \psi\|^2_{\sii ( \partial\Omega_{a,b}) } 
\end{equation*}
obtained from Theorem \eqref{formula}.
We therefore have
  \begin{equation}\label{eq3}
  \begin{aligned}
   \lambda^2_{a,b} - m^2 &=
   \inf_{\stackrel[\psi \not= 0]{}{\psi \in \Dom(T_{a,b})}} 
   \frac{\|\partial_1\psi\|^2 +
   \|\partial_2\psi\|^2 + m \|\gamma u \|^2_{OA} + m \|\gamma u\|^2_{ OB}+ m \|\gamma u\|^2_{AB}}{\|\psi\|^2}    .
  \end{aligned}
  \end{equation}

 Combining Fubini's theorem 
 and Lemma~\ref{Poincare}, we have
%
  \begin{equation*}
  \begin{aligned}
  \|\partial_1\psi\|^2 +
   \|\partial_2\psi\|^2 +& m \|\gamma u \|^2_{OA} + m \|\gamma u\|^2_{ OB}
  + m \|\gamma u\|^2_{AB} \\ 
    &= \int_0^a \int_0^{\frac{-b x}{a} +b} |\partial_2\psi (x,y)|^2 + m |u(x, 0)|^2 + m |u(x,\frac{-b x}{a} +b)|^2 dy \, dx \\
    &\quad +  \int_0^b \int_0^{\frac{-a y}{b} +a} |\partial_1\psi (x,y)|^2 + m |u(0, y)|^2 + m |u( \frac{-a y}{b} +a,y)|^2 dx \, dy \\  
   &\geq  \int_0^a \int_0^{\frac{-b x}{a} +b} \, \frac{\arctan^2(\frac{b}{a} +\sqrt{1+\frac{b^2}{a^2}})}{ (\frac{-b x}{a} +b)^2} |\psi(x,y)|^2 dy \, dx  \\
   &\quad + \int_0^b \int_0^{\frac{-a y}{b} +a} \, \frac{\arctan^2(\frac{a}{b} +\sqrt{1+\frac{a^2}{b^2}})}{ (\frac{-a y}{b} +a)^2} |\psi(x,y)|^2 dx \, dy \,,\\ 
    &\geq \Big(\frac{\arctan^2(\frac{a}{b} +\sqrt{1+\frac{a^2}{b^2}})}{b^2} + \frac{\arctan^2(\frac{b}{a} +\sqrt{1+\frac{b^2}{a^2}})}{a^2}\Big) \|\psi\|^2   .  
   \end{aligned}
  \end{equation*}
We achieve that $$
  \lambda_1(a,b)^2 - m^2 \geq \Big(\frac{\arctan^2(\frac{a}{b} +\sqrt{1+\frac{a^2}{b^2}})}{b^2} + \frac{\arctan^2(\frac{b}{a} +\sqrt{1+\frac{b^2}{a^2}})}{a^2}\Big) > \frac{\pi^2}{16a^2}
  + \frac{\pi^2}{16b^2} .$$
This concludes the proof of the lower bound of Theorem~\ref{Thm.bounds}.

The upper bound follows by using 
the eigenfunction of the Dirichlet Laplacian 
on the right triangle~$\Omega_{1,1}$ 
(see \cite[Sec.~4.3]{AP}) 
\begin{equation*}
  \psi_o(x,y) :=
  \left[ 
  \sin( 2 \pi x) \sin(\pi y) + \sin( 2 \pi y) \sin(\pi x)
  \right]
  \begin{pmatrix}
    1 \\ 1
  \end{pmatrix}
\end{equation*}
as a trial function in~\eqref{eq3}. 
After computing the value of the Rayleigh quotient, we get
$$
  \lambda_1(a,b)^2 - m^2 
  \leq 
 \frac{5\pi^2}{2} \left(\frac{1}{a^2} + \frac{1}{b^2}\right)
 \,.
$$ 
This concludes the proof of Theorem~\ref{Thm.bounds}.
\end{proof}

Finally, we establish Corollary~\ref{Corol}.
\begin{proof}[Proof of Corollary~\ref{Corol}]
By scaling, 
we can assume, without loss of generality,
that the double area equals~$1$ 
and the perimeter equals~$ 2+\sqrt{2}$.
These values correspond to the area and the perimeter 
of the isosceles right triangle in case $k=1$, respectively.
 
For the area constraint $ab=1$, we take $b := 1/a$. 
using the Theorem~\ref{Thm.bounds}, we have 
$ \lambda_{1,1} - m^2 \leq 5 \pi^2$. 
If one side length $a$ just satisfies  the condition
$$ 
 \frac{1}{a^2} + a^2 \geq 80
 ,
$$
 then Conjecture~\ref{Conj.main} will be satisfied 
 among all right triangles $\Omega_{a,b}$.  
It is easy to see that when $a \geq  9$ or $ a \leq \frac{1}{9}$ the condition holds. 
This establishes Conjecture~\ref{Conj.main}~(i).

For the perimeter constraint, we take $b:=2 +\sqrt{2}-a$ 
and restrict ourselves to $a \in (0,2+ \sqrt{2})$.
 By the similar arguments as above, 
 we arrive at the sufficient condition 
 $$ \frac{1}{a^2} + \frac{1}{(2+\sqrt{2}-a)^2} \geq 80
  $$
to have the desired inequality $\lambda_1 (a,b) \geq \lambda_1 (1,1)$.
It is not hard to see that 
the inequality holds provided that 
$ a \geq 3.5 $ or $ a \leq \frac{1}{9}$.
This establishes Conjecture~\ref{Conj.main}~(ii).

Moreover, without using scaling, we deduce that when 
$$\frac{\arctan^2(\frac{a}{b} +\sqrt{1+\frac{a^2}{b^2}})}{b^2} + \frac{\arctan^2(\frac{b}{a} +\sqrt{1+\frac{b^2}{a^2}})}{a^2} \geq  \frac{5\pi^2}{k^2} $$
happens then Conjecture~\ref{Conj.main}~(i) and (ii) hold.
\end{proof}
\begin{Remark}\label{Remark 4}
Taking Remark \eqref{1} into account, we choose $$ m = m_0(L= a)= \frac{\pi}{2 a} \sqrt{\frac{\alpha_0^2}{1-\alpha_0^2}}$$ then $\lambda_1^2(a,b) - m_0^2 \geq \frac{\pi^2}{4 a^2} + \frac{\pi^2}{16 b^2}$ . By non-decreasing property of $F_1(m)$, we obtain that \\
$$\lambda_1^2(a,b) - m^2 \geq \frac{\pi^2}{4 a^2} + \frac{\pi^2}{16 b^2}$$
for all $ m \geq \frac{\pi}{2 a} \sqrt{\frac{\alpha_0^2}{1-\alpha_0^2}}$. It gives a better estimate for the lower bound.
\end{Remark}
A direct consequence obtained from Remark \eqref{Remark 4} is the following corollary which extends the range of the side long $a$ for the validity of Conjecture~\ref{Conj.main}.
\begin{Corollary}
Let $k$ be defined as in Conjecture~\ref{Conj.main}
 and $ m\geq \frac{\pi}{2 a} \sqrt{\frac{\alpha_0^2}{1-\alpha_0^2}}$. \\
Conjecture~\ref{Conj.main}.(i) holds 
under the following extra hypotheses:
\begin{center}
$
  a \geq 9 \, k $ \quad or \quad $ a \leq \frac{k}{5}
$, \\
\end{center}
Conjecture~\ref{Conj.main}.(ii) holds 
under the following extra hypotheses:
\begin{center}
$
\displaystyle
  a \geq  3.5 \, k $ \quad or \quad $ a \leq \frac{k}{5}
$.\\
\end{center}

\end{Corollary}
\appendix\section{Proof of Remark 1}\label{A}
In this section we give the proof of Remark 1, 
extending the validity of the formula \eqref{norm}
to polygons.
We consider a polygon $OBCDA $ with coordinates  
$A:=(a,0), B:=(b,c), C:=(c_1,c_2),
 D:=(d_1,d_2)$, see Figure~\ref{Fig2}.
Without loss of generality, we can suppose $0<b<a< c_1<d_1$.
 The proof for the general polygon is
 the same as in this case. Dividing the polygon into triangles is a crucial step to achieve the proof.

\begin{figure}[h]\label{pentagon}
  \begin{center}
  \includegraphics[scale=1.5]{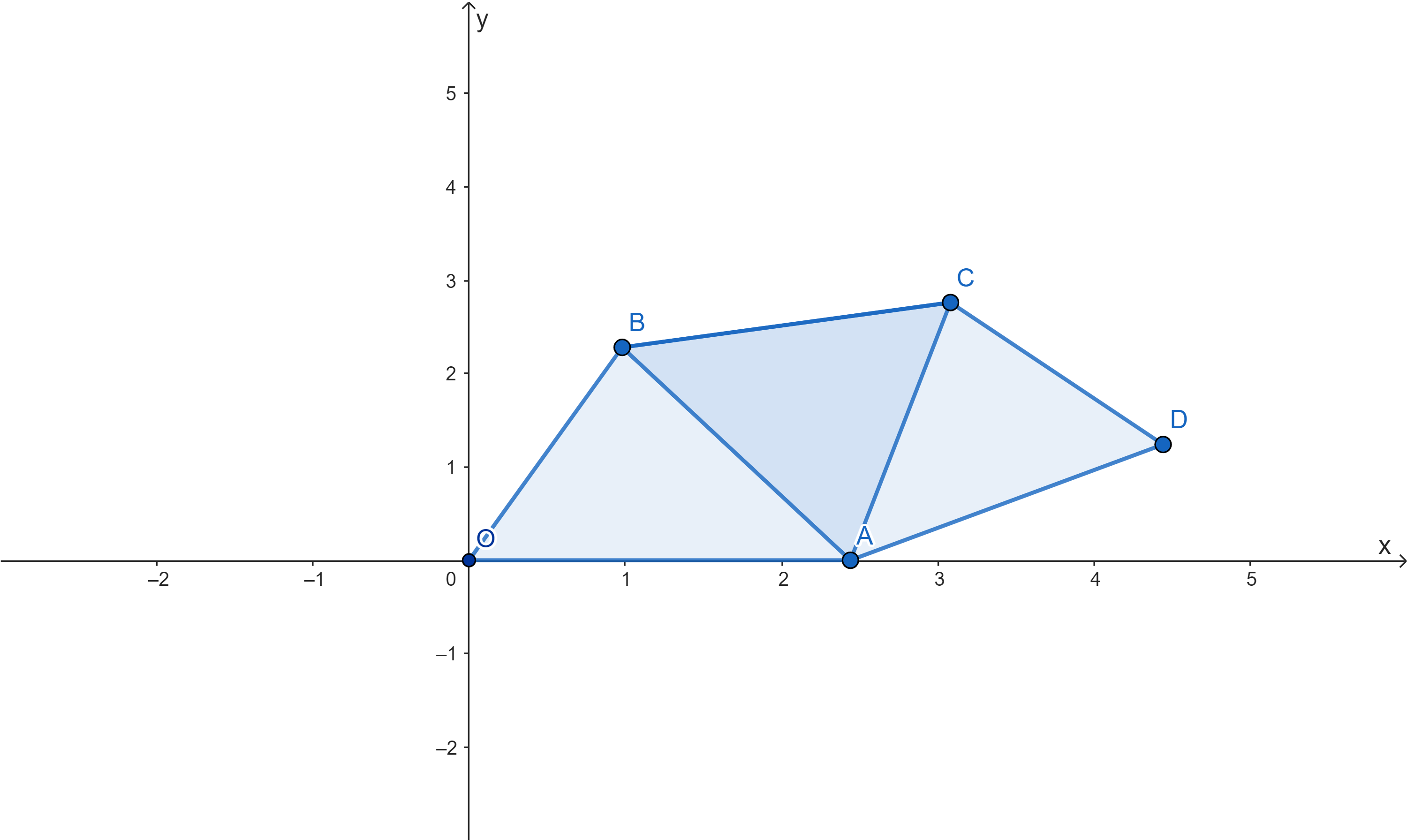}
  \caption{The pentagon $OBCDA$.
  \\
  }
  \label{Fig2}
  \end{center}
\end{figure}

\begin{Theorem} 
\begin{equation*}  
  \|T u\|^2   = \|\nabla u\|^2_{\sii(OBCDA)} + m^2 \|u\|^2_{\sii(OBCDA)} + m \, \|\gamma u\|^2_{\sii ( \partial (OBCDA)) } .
\end{equation*}
\end{Theorem}
\begin{proof} Using integration by parts 
and the density arguments described in Theorem~\ref{formula}, we compute the norm $\|Tu\|^2$ in every triangle divided. Firstly, we apply on the triangle $OAB$
%
\begin{equation*}
\begin{aligned}
\|Tu\|^2 &= \|m u_1 -i(\partial_1 -i\partial_2) u_2 \|^2 + \|i(\partial_1+ i\partial_2) u_1 + m u_2 \|^2\\
&= m^2 \|u\|^2 +\|\nabla u\|^2 -2\Re(\partial_1 u_2,i\partial_2 u_2) +2\Re(\partial_1 u_1, i\partial_2 u_1)\\
& \quad -2 m\Re(u_1, i(\partial_1-i\partial_2)u_2) - 2 m i\Re((\partial_1+i\partial_2)u_1, u_2)\\
&= m^2 \|u\|^2 +\|\nabla u\|^2 -2\Re(\partial_1 u_2,i\partial_2 u_2) +2\Re(\partial_1 u_1, i\partial_2 u_1)\\
& \quad -2 \Re m i\int_0^c \overline{u_1} u_2|_{x=\frac{b}{c}y}^{x=\frac{y(b-a)}{c}+a} dy -2\Re m \int_0^b \overline{u_1} u_2|_{y=0}^{y=\frac{cx}{b}} dx -2\Re m \int_b^a \overline{u_1} u_2|_{y=0}^{y=\frac{cx-ac}{b-a}} dx \\
&= m^2 \|u\|^2 +\|\nabla u\|^2 -2\Re(\partial_1 u_2,i\partial_2 u_2) +2\Re(\partial_1 u_1, i\partial_2 u_1)+ m \|u\|^2_{OA} + m \|u\|^2_{OB}\\
& \quad -2 \Re m i\int_0^c \overline{u_1} u_2|_{x=\frac{y(b-a)}{c}+a} dy -2 \Re m \int_b^a \overline{u_1} u_2|_{y=\frac{cx-ac}{b-a}} dx\\
&= m^2 \|u\|^2 +\|\nabla u\|^2 -2\Re(\partial_1 u_2,i\partial_2 u_2) +2\Re(\partial_1 u_1, i\partial_2 u_1)+ m \|u\|^2_{OA} + m \|u\|^2_{OB}\\
&\quad + \int_{AB} -2 m \Re  i \, n_{1_{AB}}\overline{u_1} u_2 - 2 m \Re  \, n_{2_{AB}} \overline{u_1} u_2.
\end{aligned}
\end{equation*}
%
Putting 
$ I_1 := - 2\Re(\partial_1 u_2,i\partial_2 u_2)$
and $ I_2 := 2\Re(\partial_1 u_1, i\partial_2 u_1)$,
we compute
%
\begin{equation*}
\begin{aligned}
I_1 &= 2\Re (u_2, i \partial_1\partial_2 u_2) - 2\Re \int_0^c \overline{u_2} i \partial_2 u_2|_{x=\frac{by}{c}}^{x=\frac{y(b-a)}{c}+a} dy \\
&= -2 \Re( \partial_2 u_2, i \partial_1 u_2)- 2\Re \int_0^c \overline{u_2} i \partial_2 u_2|_{x=\frac{by}{c}}^{x=\frac{y(b-a)}{c}+a} dx \\
& \quad + 2\Re \int_0^b \overline{u_2} i \partial_1 u_2|_{y=0}^{y=\frac{cx}{b}} dx + 2\Re \int_{b}^a \overline{u_2} i \partial_2 u_2|_{y=0}^{y= \frac{cx-ac}{b-a}} dx\\
&= -I_1 - 2\Re \int_0^c \overline{u_2} i \partial_2 u_2|_{x=\frac{by}{c}}^{x=\frac{y(b-a)}{c}+a} dx \\
& \quad + 2\Re \int_0^b \overline{u_2} i \partial_1 u_2|_{y=0}^{y=\frac{cx}{b}} dx + 2\Re \int_{b}^a \overline{u_2} i \partial_2 u_2|_{y=0}^{y= \frac{cx-ac}{b-a}} dx .
\end{aligned}
\end{equation*}
Therefore, we get 
$$I_1 = -\Re i \int_0^c \overline{u_2}  \partial_2 u_2|_{x=\frac{by}{c}}^{x=\frac{y(b-a)}{c}+a} dy  + \Re i \int_0^b \overline{u_2}  \partial_1 u_2|_{y=0}^{y=\frac{cx}{b}} dx + \Re i \int_{b}^a \overline{u_2}  \partial_2 u_2|_{y=0}^{y= \frac{cx-ac}{b-a}} dx .$$
An analogous computation gives 
$$I_2 = \Re i \int_0^c \overline{u_1}  \partial_2 u_1|_{x=\frac{by}{c}}^{x=\frac{y(b-a)}{c}+a} dx  - \Re i \int_0^b \overline{u_1}  \partial_1 u_1|_{y=0}^{y=\frac{cx}{b}} dx - \Re i \int_{b}^a \overline{u_1}  \partial_2 u_1|_{y=0}^{y= \frac{cx-ac}{b-a}} dx.$$
In addition, on the side $OB$, we have $u_2 = \frac{-b-ci}{\sqrt{b^2+c^2}} u_1$ and
\begin{equation*}
\begin{aligned}
\int_0^b \overline{u_2}\partial_1 u_2 |_{y=\frac{cx}{b}} dx & = \int_0^b \frac{-b+ci}{\sqrt{b^2+c^2}} \overline{u_1} \partial_1 u_2|_{y=\frac{cx}{b}} dx = - \int_0^b \frac{-b+ci}{\sqrt{b^2+c^2}}\overline{\partial_1 u_1} \frac{-b-ci}{\sqrt{b^2+c^2}} u_1|_{y=\frac{cx}{b}} dx\\
&= - \int_0^b \overline{\partial_1 u_1} u_1|_{y=\frac{cx}{b}} dx .
\end{aligned}
\end{equation*}
As a result, we obtain
$$\Re i \int_0^b (\overline{ u_2}\partial_1u_2-\overline{ u_1} \partial_1u_1)|_{y=\frac{cx}{b}} dx =0 . $$
Similarly, we also have
$$\Re i \int_0^c (\overline{ u_2} \partial_2 u_2-\overline{ u_1} \partial_2 u_1)|_{x=\frac{by}{c}} dx =0$$
and
$$\Re i \int_0^a (\overline{ u_2} \partial_1 u_2-\overline{ u_1} \partial_1 u_1)|_{y=0} dx =0. $$
Hence, \\
$$
I_1+I_2 = -\Re i \int_0^c (\overline{u_2}  \partial_2 u_2 -\overline{u_1}  \partial_2 u_1) |_{x=\frac{y(b-a)}{c}+a} dy + \Re i \int_b^a (\overline{u_2}  \partial_1 u_2 -\overline{u_1}  \partial_1 u_1)|_{y= \frac{cx-ac}{b-a}} ,
$$
which is equivalent to
$$ I_1+I_2 =   -\Re i n_{1_{AB}} \int_{AB} (\overline{u_2}  \partial_2 u_2 -\overline{u_1}  \partial_2 u_1)                             
+ \Re i n_{2_{AB}} \int_{AB} (\overline{u_2}  \partial_1 u_2 -\overline{u_1}  \partial_1 u_1).
$$
In summary, the norm computed on the triangle $OAB$ reads
\begin{equation*}
\begin{aligned}
\|Tu\|_{L^2(OAB)}^2 &= m^2 \|u\|_{L^2(OAB)}^2 +\|\nabla u\|_{L^2(OAB)}^2+ m \|\gamma u\|^2_{OA} + m \|\gamma u\|^2_{OB} + \int_{AB} (-2 m \Re  i \, n_{1_{AB}} \overline{u_1} u_2 \\
& \qquad- 2 m \Re  \, n_{2_{AB}} \overline{u_1} u_2 )
 -\Re i n_{1_{AB}} \int_{AB} (\overline{u_2}  \partial_2 u_2 -\overline{u_1}  \partial_2 u_1)                             
+ \Re i n_{2_{AB}} \int_{AB} (\overline{u_2}  \partial_1 u_2 -\overline{u_1}  \partial_1 u_1).                 
\end{aligned}
\end{equation*}
 We suppose that $n$ is the outward unit vector in each triangle divided. This is the same notation but it is different depending on each triangle. By analogous computations, we obtain the square of the operator defined on the other triangles as follows,
\begin{equation*}
\begin{aligned}
\|Tu\|_{L^2(ABC)}^2& = m^2 \|u\|_{L^2(ABC)}^2 +\|\nabla u\|_{L^2(ABC)}^2+ m \|\gamma u\|^2_{BC}  + \int_{AB} (-2 m \Re  i \, n_{1_{AB}} \overline{u_1} u_2 - 2 m \Re  \, n_{2_{AB}} \overline{u_1} u_2)\\
& \qquad -\Re i n_{1_{AB}} \int_{AB} (\overline{u_2}  \partial_2 u_2 -\overline{u_1}  \partial_2 u_1)                             
+ \Re i n_{2_{AB}} \int_{AB} (\overline{u_2}  \partial_1 u_2 -\overline{u_1}  \partial_1 u_1) \\
& \qquad+ \int_{AC} (-2 m \Re  i \, n_{1_{AC}} \overline{u_1} u_2 - 2 m \Re  \, n_{2_{AC}} \overline{u_1} u_2)
 -\Re i n_{1_{AC}} \int_{AC} (\overline{u_2}  \partial_2 u_2 -\overline{u_1}  \partial_2 u_1)  \\                           
& \qquad+ \Re i n_{2_{AC}} \int_{AC} (\overline{u_2}  \partial_1 u_2 -\overline{u_1}  \partial_1 u_1)                
\end{aligned}
\end{equation*}
and 
\begin{equation*}
\begin{aligned}
\|Tu\|_{L^2(ADC)}^2& = m^2 \|u\|_{L^2(ADC)}^2 +\|\nabla u\|_{L^2(ADC)}^2+ m \|\gamma u\|^2_{DC} + m \|\gamma u\|^2_{DA} + \int_{AC} (-2 m \Re  i \, n_{1_{AC}} \overline{u_1} u_2 \\
&- 2 m \Re  \, n_{2_{AC}} \overline{u_1} u_2)
-\Re i n_{1_{AC}} \int_{AC} (\overline{u_2}  \partial_2 u_2 -\overline{u_1}  \partial_2 u_1)                             
+ \Re i n_{2_{AC}} \int_{AC} (\overline{u_2}  \partial_1 u_2 -\overline{u_1}  \partial_1 u_1)                .
\end{aligned}
\end{equation*}

When dividing the polygons into triangles, we deduce that the outward normal in the inner sides in the adjacent triangles are opposite then the integration computed in these sides 
will be canceled. Summarising these computations, we obtain the following formula:
\begin{equation*}
\begin{aligned}
\|Tu\|_{L^2(OBCDA)}^2 &= m^2 \|u\|_{L^2(OBCDA)}^2 +\|\nabla u\|_{L^2(OBCDA)}^2+ m \|\gamma u\|^2_{OA} + m \|\gamma u\|^2_{OB}\\
&+ m \|\gamma u\|^2_{DA} + m \|\gamma u\|^2_{CB} + + m \|\gamma u\|^2_{CD} .                 
\end{aligned}
\end{equation*}
Therefore, the proof for the square of the operator is completed.

\end{proof}


\subsection*{Acknowledgment}
We are grateful to David Krej\v{c}i\v{r}{\'\i}k for useful discussions. The author was supported by the EXPRO grant No.~20-17749X
of the Czech Science Foundation.

%

\begin{thebibliography}{10}
\bibitem{AP} Anil Damle and Geoffrey Colin Peterson, \emph{Understanding the eigenstructure of various triangles}, SIAM Undergraduate Research Online \textbf{3} (2010), 187-208.
\bibitem{2DK} Ph. Briet and D.~Krej{\v{c}}i{\v{r}}{\'i}k, \emph{Spectral optimisation of {D}irac rectangles}, J. Math. Phys. \textbf{63} (2022) 013502.
\bibitem{Antunes-Benguria-Lotoreichik-Ourmires-Bonafos_2021}
P.~R.~S. Antunes, R.~Benguria, V.~Lotoreichik, and T.~Ourmi{\`e}res-Bonafos,
  \emph{A variational formulation for {D}irac operators in bounded domains.
  {A}pplications to spectral geometric inequalities}, Comm. Math. Phys.
  \textbf{386} (2021), 781--818.

\bibitem{AFK}
P.~R.~S. Antunes, P.~Freitas and D.~Krej\v{c}i\v{r}\'{\i}k, 
\emph{Bounds and extremal domains for Robin eigenvalues with negative boundary parameter},
Adv. Calc. Var. \textbf{10} (2017), 357--380. 
  

\bibitem{Arrizibalaga-LeTreust-Mas-Raymond_2019}
N.~Arrizibalaga, L.~Le~Treust, A.~Mas, and N.~Raymond, \emph{The {MIT} bag
  model as an infinite mass limit}, J. {\'E}c. Polytech. Math. \textbf{6}
  (2019), 329--365.

\bibitem{Arrizibalaga-LeTreust-Raymond_2017}
N.~Arrizibalaga, L.~Le~Treust, and N.~Raymond, \emph{On the {MIT} bag model in
  the non-relativistic limit}, Comm. Math. Phys. \textbf{354} (2017), 641--669.

\bibitem{Arrizibalaga-LeTreust-Raymond_2018}
N.~Arrizibalaga, L.~Le~Treust, and N.~Raymond, \emph{Extension operator for the {MIT} bag model}, Ann. Fac. Sci.
  Toulouse Math. 
  \textbf{29} (2020) 135--147.

\bibitem{Barbaroux-Cornean-LeTreust-Stockmeyer_2019}
J.-M. Barbaroux, H.~D. Cornean, L.~Le~Treust, and E.~Stockmeyer,
  \emph{Resolvent convergence to {D}irac operators on planar domains}, Ann.
  Henri Poincar{\'e} \textbf{20} (2019), 1877--1891.

\bibitem{Behrndt}
 J.~Behrndt, M.~Holzmann, Ch.~Stelzer, and G.~Stenzel,
  \emph{A class of singular perturbations of the Dirac operator:
boundary triplets and Weyl functions}, 15--35;
In Contributions to Mathematics and Statistics,
Essays in honor of Seppo Hassi,
H.S.V. de Snoo \& H.L. Wietsma (Eds.),
Acta Wasaensia 462, University of Vaasa, 2021.
 
\bibitem{Benguria-Fournais-Stockmeyer-Bosch_2017b}
R.~D. Benguria, S.~Fournais, E.~Stockmeyer, and H.~Van Den~Bosch,
  \emph{Self-adjointness of two-dimensional {D}irac operators on domains}, Ann.
  Henri Poincar{\'e} \textbf{18} (2017), 1371--1383.

\bibitem{Benguria-Fournais-Stockmeyer-Bosch_2017}
\bysame, \emph{Spectral gaps of {D}irac operators describing graphene quantum
  dots}, Math. Phys. Anal. Geom. \textbf{20} (2017), 11.

\bibitem{Bogosel-Bucur}
B.~Bogosel and D.~Bucur, On the polygonal Faber-Krahn inequality,
arXiv:2203.16409  [math.OC] (2022).

\bibitem{BBKO}
W.~Borrelli, Ph. Briet, Krej{\v{c}}i{\v{r}}{\'i}k, and
  T.~Ourmi{\`e}res-Bonafos, \emph{Spectral properties of relativistic quantum
  waveguides}, 
  Ann. Henri Poincar\'e \textbf{23} (2022) 4069--4114. 

  
\bibitem{FK7}
P.~Freitas and D.~Krej\v{c}i\v{r}\'{\i}k, \emph{The first {R}obin eigenvalue
  with negative boundary parameter}, Adv. Math. \textbf{280} (2015), 322--339.  

\bibitem{Henrot}
A.~Henrot, \emph{Extremum problems for eigenvalues of elliptic operators},
  Birkh{\"a}user, Basel, 2006.
  
\bibitem{Henrot2}
\bysame, \emph{Shape optimization and spectral theory}, De Gruyter, Warsaw,
  2017.  
  
\bibitem{Indrei}
E.~Indrei, \emph{On the first eigenvalue of the Laplacian for polygons}, 
arXiv:2210.14806 [math.AP] (2022).

\bibitem{private}
D.~Krej\v{c}i\v{r}\'ik, private communication, November 2022. 

\bibitem{KLV}
D. Krej\v{c}i\v{r}\'ik, V. Lotoreichik and T. Vu,
\emph{Reverse isoperimetric inequality for the lowest Robin eigenvalue 
of a triangle},
arXiv:2204.03235 [math.OC] (2022). 

\bibitem{Laugesen_2019}
R.~S. Laugesen, \emph{The {R}obin {L}aplacian --- spectral conjectures,
  rectangular theorems}, J. Math. Phys. \textbf{60} (2019), 121507.

\bibitem{LeTreust-Ourmieres-Bonafos_2018}
L.~Le~Treust and T.~Ourmi\`eres-Bonafos, \emph{Self-adjointness of {D}irac
  operators with infinite mass boundary conditions in sectors}, Ann. H.
  Poincar{\'e} \textbf{19} (2018), 1465--1487.

\bibitem{Lotoreichik-Ourmieres_2019}
V.~Lotoreichik and T.~Ourmi{\`e}res-Bonafos, \emph{A sharp upper bound on the
  spectral gap for graphene quantum dots}, Math. Phys. Anal. Geom. \textbf{22}
  (2019), 13.

\bibitem{AIM-2019}
{Problem List of the AIM Workshop (D.~Krej\v{c}i\v{r}\'ik, S.~Larson, and
  V.~Lotoreichik, eds.)}, \emph{Shape optimization with surface interactions},
  \texttt{http://aimpl.org/shapesurface/}, 2019, {S}an Jose, USA.

\end{thebibliography}
%


\providecommand{\bysame}{\leavevmode\hbox to3em{\hrulefill}\thinspace}
\providecommand{\MR}{\relax\ifhmode\unskip\space\fi MR }
\providecommand{\MRhref}[2]{%
  \href{http://www.ams.org/mathscinet-getitem?mr=#1}{#2}
}
\providecommand{\href}[2]{#2}

\end{document}